\documentclass{amsart}

\usepackage{adjustbox}
\usepackage{amsaddr}
\usepackage{amsmath}
\usepackage{amssymb}
\usepackage{bm}
\usepackage{booktabs}
\usepackage[margin=1in]{geometry}
\usepackage{hyperref}
\usepackage{cleveref} 
\usepackage{siunitx}
\usepackage{stmaryrd}
\usepackage{tikz}

\usepackage[giveninits,date=year,sortcites]{biblatex}
\newcolumntype{Y}{S[table-format=5.0]}
\newcolumntype{Z}{S[table-format=6.0]}
\newcolumntype{?}{!{\vrule width 1pt}}

\newcommand{\tabstrut}{\rule{0pt}{0.5cm} \rule[-0.25cm]{0pt}{0pt}}

\theoremstyle{plain}

\newtheorem{lem}{Lemma}
\newtheorem*{lem*}{Lemma}
\newtheorem{thm}{Theorem}

\theoremstyle{remark}
\newtheorem{rem}{Remark}

\crefname{thm}{theorem}{theorems}

\newcommand{\Hdiv}{\bm{H}(\operatorname{div})}
\newcommand{\HdivOmega}{\bm{H}(\operatorname{div},\Omega)}

\newcommand{\T}{\mathcal{T}}
\newcommand{\kk}{\kappa}
\newcommand{\khat}{\widehat{\kappa}}
\newcommand{\nn}{\bm{n}}
\newcommand{\uu}{\bm{u}}
\newcommand{\vv}{\bm{v}}
\newcommand{\ff}{\bm{f}}

\newcommand{\Vh}{\bm{V}_h}
\newcommand{\VV}{\bm{V}}
\newcommand{\Vhat}{\widehat{\VV}_h}

\newcommand{\WW}{\bm{W}}
\newcommand{\Wh}{\bm{W}_h}

\newcommand{\QQ}{\mathcal{Q}}

\newcommand{\Qhat}{\widehat{Q}}

\newcommand{\II}{\mathcal{I}}

\newcommand{\llb}{\llbracket}
\newcommand{\rrb}{\rrbracket}

\newcommand{\tr}{T}

\newcommand{\Bsub}{B_{\mathrm{sub}}}
\newcommand{\Bfic}{B_{\mathrm{fic}}}
\newcommand{\Baux}{B_{\mathrm{aux}}}

\newcommand{\oo}[1]{{\otimes #1}}

\def\iii{\vert\kern-0.25ex\vert\kern-0.25ex\vert}

\DeclareMultiCiteCommand{\semicites}[\mkbibbrackets]{\cite}{\addsemicolon\space}

\SetSymbolFont{stmry}{bold}{U}{stmry}{m}{n}

\makeatletter
\newcommand*{\newreptext}[1]{%
  \begingroup 
    \csname @safe@actives@true\endcsname
  \expandafter\endgroup
  \expandafter\newcommand\csname reptext@#1\endcsname
}

\newcommand*{\reptext}[1]{%
  \begingroup
  \csname @safe@actives@true\endcsname 
  \@ifundefined{reptext@#1}{%
    \@latex@error{\string\reptext{#1} is undefined}\@ehc
    \endgroup
    \textbf{??}%
  }{%
    \endgroup
    \@nameuse{reptext@#1}%
  }%
}
\makeatother

\begin{document}

\title[Preconditioners for IPDG in $\Hdiv$]{Subspace and auxiliary space preconditioners for high-order interior penalty discretizations in $\Hdiv$}
\author{Will Pazner}
\address{Fariborz Maseeh Department of Mathematics and Statistics, Portland State University, Portland, OR}

\begin{abstract}
   In this paper, we construct and analyze preconditioners for the interior penalty discontinuous Galerkin discretization posed in the space $\Hdiv$.
   These discretizations are used as one component in exactly divergence-free pressure-robust discretizations for the Stokes problem.
   Three preconditioners are presently considered: a subspace correction preconditioner using vertex patches and the lowest-order $H^1$-conforming space as a coarse space, a fictitious space preconditioner using the degree-$p$ discontinuous Galerkin space, and an auxiliary space preconditioner using the degree-$(p-1)$ discontinuous Galerkin space and a block Jacobi smoother.
   On certain classes of meshes, the subspace and fictitious space preconditioners result in provably well-conditioned systems, independent of the mesh size $h$, polynomial degree $p$, and penalty parameter $\eta$.
   All three preconditioners are shown to be robust with respect to $h$ on general meshes, and numerical results indicate that the iteration counts grow only mildly with respect to $p$ in the general case.
   Numerical examples illustrate the convergence properties of the preconditioners applied to structured and unstructured meshes.
   These solvers are used to construct block-diagonal preconditioners for the Stokes problem, which result in uniform convergence when used with MINRES.
\end{abstract}

\maketitle

\section{Introduction}

This paper is concerned with the development of subspace and auxiliary space preconditioners for interior penalty discontinuous Galerkin (DG) discretizations posed in the space $\Hdiv$.
Such discretizations appear, for instance, as components in $\Hdiv$-conforming discretizations for Stokes and incompressible Navier--Stokes equations, which possess desirable properties, such as exactly divergence-free velocity fields \cite{Cockburn2006,Wang2009}.
Importantly, these methods are pressure-robust, meaning that error estimates for the velocity are decoupled from those for the pressure \cite{John2017}.
Solvers for the Stokes saddle-point problem can be constructed using block-diagonal or block-triangular preconditioners \cite{Benzi2005}.
These block preconditioners in turn require approximations to the inverse of the vector Laplacian operator.
The present paper is concerned with the development of such approximations for the interior penalty DG vector Laplacian in $\Hdiv$.

There is a large literature on the construction of preconditioners for discontinuous Galerkin discretizations, including multigrid methods (geometric multigrid and $p$-multigrid) \cite{Gopalakrishnan2003,Kanschat2004,Helenbrook2006}, algebraic multigrid (AMG) methods \cite{Helenbrook2006,Antonietti2020a}, and domain decomposition methods \cite{Antonietti2013,Antonietti2010}.
Subspace correction preconditioners using $H^1$-conforming subspaces have also been studied \cite{Antonietti2016,Pazner2020a,Dobrev2006}.
Recently, there has been significant interest in the development of matrix-free methods for DG discretizations, which are particularly important for high-order discretizations \cite{Kronbichler2019a,Bastian2019,Pazner2018e,Pazner2024}.

In the present work, we are interested in developing preconditioners for DG discretization posed in $\Hdiv$ finite element spaces.
Although the bilinear forms are the same as in the standard DG case, the use of Piola-mapped Raviart--Thomas finite elements results in some additional complexities that require special treatment.
We consider the construction of three preconditioners.
The first is a domain decomposition preconditioner using vertex patches and the lowest-order $H^1$-conforming space as a coarse space.
The use of such high-order vertex patches is traditional, and was introduced by Pavarino \cite{Pavarino1992,Pavarino1993}.
Second, we consider a fictitious space preconditioner, where a related problem is solved on a so-called fictitious space;
in our case, the fictitious space is the standard discontinuous Galerkin space, which on meshes containing only affinely transformed elements is a superset of the $\Hdiv$ space.
Finally, we consider an auxiliary space preconditioner, which combines an auxiliary space (in our case, a lower-degree discontinuous Galerkin space), with a smoother, which is obtained by solving a set of independent small problems corresponding to non-overlapping patches.

For each of these preconditioners, we analyze the conditioning of the preconditioned system.
Stronger bounds can be proven in the case of affine elements, for which the Piola transformation is spatially constant, and the Raviart--Thomas basis functions are polynomials.
In this case, it can be shown that the subspace correction preconditioner results in a system whose condition number is independent of mesh size $h$, polynomial degree $p$, and DG penalty parameter $\eta$.
Constants bounding the operators in the case of the fictitious space and auxiliary space preconditioners are also shown to be independent of discretization parameters.
On general meshes, the preconditioners are shown to be independent of mesh refinement, and the dependence on $p$ and $\eta$ is studied numerically.
Only relatively mild dependence on $p$ is observed numerically.

The structure of the paper is as follows.
In \Cref{sec:discretization}, the model problem and DG discretization are described;
some well-known eigenvalue estimates for the DG operator are recalled, and properties of the Gauss--Lobatto basis are discussed.
In \Cref{sec:preconditioning}, the three preconditioners discussed above are formulated and analyzed.
Bounds on the condition number of the preconditioned system are presented, with an focus on the cases in which $p$- and $\eta$-independent conditioning can be obtained.
\Cref{sec:technical} includes some technical results on Gauss--Lobatto interpolation that are necessary for the analysis in the preceding section.
The performance of the preconditioners is numerically studied in \Cref{sec:results}.
Several test cases are considered, including structured and unstructured meshes.
Conjugate gradient iteration counts and runtimes are presented.
Block-diagonal preconditioners for the Stokes saddle-point system are also studied.
Conclusions are presented in \Cref{sec:conclusions}.

\section{Model Problem and Discretization}
\label{sec:discretization}

Let $\Omega \subseteq \mathbb{R}^d$ ($d \in \{1,2,3\})$ denote the spatial domain.
We consider the Poisson problem for a vector-valued unknown $\uu : \Omega \to \mathbb{R}^d$,
\begin{equation}
   \label{eq:poisson}
   \left\{~
   \begin{aligned}
      -\nabla \cdot (b \nabla \uu) &= \ff && \text{in $\Omega$}, \\
      \uu &= \bm{g}_D && \text{on $\Gamma_D$}, \\
      \frac{\partial \uu}{\partial \nn} &= \bm{g}_N && \text{on $\Gamma_N$},
   \end{aligned}\right.
\end{equation}
where $b$ is a given spatially varying scalar coefficient, $\ff \in [L^2(\Omega)]^d$, $\partial\Omega = \Gamma_D \cup \Gamma_N$, and $\bm{g}_D$ and $\bm{g}_N$ are specified Dirichlet and Neumann boundary conditions, respectively.
For simplicity of exposition, we will take $b \equiv 1$, $\Gamma_D = \partial\Omega$ and $\bm{g}_D = 0$;
the extension to the more general case is standard.

This equation is discretized using an interior penalty discontinuous Galerkin (IPDG) discretization on an $\Hdiv$-conforming finite element space.
Let $\T$ denote a mesh of the domain $\Omega$, consisting of elements $\kk$, where each element is the image of the reference cube $\khat = [0,1]^d$ under a transformation mapping $T_\kk : \khat \mapsto \kk$.
In this work, we take $T_\kk$ to be a continuous, piecewise polynomial of degree at most 1 in each variable;
the elements $\kk$ are straight-sided quadrilaterals.
In the general case, $T_\kk$ is not restricted to be affine, and so the elements are not necessarily parallelepipeds, and may by skewed.
The Jacobian matrix of $T_\kk$ is denoted $J_\kk$.
In the case of affine mappings, the entries of $J_\kk$ are constant; in the general case they may be spatially varying.
Let $\QQ_{p}$ denote the space of polynomials of at most degree $p$.
Similarly, let $\QQ_{p,q}$ denote the space of bivariate polynomials of degree at most $p$ in the first variable, and degree at most $q$ in the second variable;
likewise for $\QQ_{p,q,r}$ in three variables.
On the reference element, the local (reference) Raviart--Thomas space $\VV(\khat)$ is given by
\begin{equation}
   \label{eq:reference-space}
   \VV(\khat) = \begin{cases}
      \QQ_p & (d=1)\\
      \QQ_{p,p-1} \times \QQ_{p-1,p}  & (d=2)\\
      \QQ_{p,p-1,p-1} \times \QQ_{p-1,p,p-1} \times \QQ_{p-1,p-1,p}  & (d=3)
   \end{cases}
\end{equation}
Let $\Vh$ denote the Raviart--Thomas finite element space defined on the mesh $\T$,
\[
   \Vh = \left\{
      \vv_h \in \HdivOmega : \vv_h|_\kk \circ T_\kk \in \VV(\kk)
   \right\},
\]
where the local space $\VV(\kk)$ is the image of the reference space $\VV(\khat)$ under the $\Hdiv$ Piola transformation,
\begin{equation}
   \label{eq:piola}
   \VV(\kk) = \det(J_\kk)^{-1} J_\kappa \VV(\khat).
\end{equation}
The vector-valued functions in the space $\Vh$ possess normal continuity at element interfaces, but may have discontinuous tangential components.

Consider $\phi \in L^2(\Omega)$ such that $\phi|_\kk \in H^1(\kk)$ for all $\kk  \in \T$.
Define the element-wise gradient $\nabla_h \phi \in [L^2(\Omega)]^d$ such that $(\nabla_h \phi)|_\kk = \nabla (\phi|_\kk$).
Let $\Gamma$ denote the mesh skeleton, i.e.~$\Gamma = \bigcup_{\kk \in \T} \partial \kk$.
For a given face $e \subset \Gamma$ bordering two elements $\kk^+$ and $\kk^-$, let $\phi^+$ denote the trace of $\phi$ on $e$ from within $\kk^+$, and likewise, let $\phi^-$ denote the trace of $\phi$ from within $\kk^-$.
Similarly, let $\bm{n}^+$ denote the vector normal to $e$ pointing outwards from $\kk^+$, and let $\bm{n}^-$ denote the normal vector pointing outwards from $\kk^-$ (i.e.~$\bm{n}^- = - \bm{n}^+$).
On such a face $e$, the average $\{ \phi \}$ and jump $\llb \phi \rrb$ are defined by
\[
   \{ \phi \} = \frac{1}{2}\left( \phi^+ + \phi^- \right), \qquad
   \llb \phi \rrb = \phi^+ \bm n^+ + \phi^- \bm n^-.
\]
Note that the jump of a scalar is a vector;
extending this definition to $\vv_h \in V_h$, the jump of a vector-valued function is the matrix
\[
   \llb \bm\sigma \rrb = \bm \sigma^+ \otimes \bm n^+ + \bm \sigma^- \otimes \bm n^-,
\]
where $\otimes$ denotes the outer product.

We discretize \eqref{eq:poisson} by looking for $\uu_h \in \Vh$ such that
\begin{equation}
   \label{eq:variational-problem}
   a(\uu_h, \vv_h) = (\ff, \vv_h) \qquad \text{for all $\vv_h \in \Vh$},
\end{equation}
where $(\cdot\,,\cdot)$ denotes the $L^2$ inner product over $\Omega$, and $a(\cdot\,,\cdot)$ is the interior penalty bilinear form,
\begin{equation}
   \label{eq:ipdg}
   a(\uu_h, \vv_h) =
      (\nabla_h \uu_h, \nabla_h \vv_h)
      - \langle \llb \uu_h \rrb, \{ \nabla_h \vv_h \} \rangle
      - \langle \{ \nabla_h \uu_h \}, \llb \vv_h \rrb \rangle
      + \langle \alpha \llb \uu_h \rrb, \llb \vv_h \rrb  \rangle.
\end{equation}
In the above, $\langle \cdot, \,\cdot \rangle$ represents the $L^2$ inner product over the mesh skeleton $\Gamma$, and $\alpha = \eta p^2 / h$, where the penalty parameter $\eta$ is a constant (independent of $h$ and $p$) that must be chosen sufficiently large to ensure stability of the method.

The discretization \eqref{eq:ipdg} is used in $\Hdiv$-conforming DG discretizations of Stokes \cite{Cockburn2006,Kanschat2015,Wang2009} and incompressible Navier--Stokes equations \cite{Fu2019}.
These methods have the attractive property that the resulting velocity field is pointwise divergence-free, which in turn results in favorable properties, such as amenability to monolithic multigrid preconditioning \cite{Kanschat2015,Cui2024}, and pressure-robust error estimates.

We emphasize that while the bilinear form \eqref{eq:ipdg} is the standard interior penalty DG (IPDG) discretization (see e.g.~\cite{Arnold1982,Arnold2002}), the method defined by \eqref{eq:variational-problem} is distinct from the IPDG method defined on the standard vector-valued DG finite element space: the problem \eqref{eq:variational-problem} is posed on the $\Hdiv$ finite element space $\Vh$, which then has important implications for preconditioning, as discussed in the following section.

\begin{rem}[Affine and non-affine elements]
The case when the mesh transformation $T_{\kk}$ is affine presents some simplifications.
In this case, the element $\kk$ is a parallelepiped, and the transformation Jacobian matrix and determinant are both constant.
Then, the Piola transformation \eqref{eq:piola} is constant, and the elements of the local space $\VV(\kk)$ are polynomials.
However, when the mesh transformation is not affine (for example, on skewed quadrilaterals; more generally, on curved elements), then the entries of the Jacobian matrix and the determinant are spatially varying (they are polynomials in the case of isoparametric elements), and the elements of the local space $\VV(\kk)$ are not generally polynomials.
The Raviart--Thomas finite element space on meshes with non-affine case can lead to some well-studied difficulties, including degraded convergence rates; see for example \cite{Arnold2005}.
\end{rem}

\begin{rem}[Notation]
   In what follows, we will use the notation $a \lesssim b$ to express that $a \leq C b$, where $C$ is a constant independent of the discretization parameters $h$, $p$, and $\eta$.
   Similarly, $a \gtrsim b$ is used to mean $b \lesssim a$, and $a \approx b$ means that both $a \lesssim b$ and $b \lesssim a$.
   Dependence of generic constants $C$ on discretization parameters $h$, $p$, or $\eta$ will be mentioned explicitly.
   The standard norms and seminorms on the Sobolev space $H^k(D)$ will be denoted $\|\cdot\|_{k,D}$ and $|\cdot|_{k,D}$, respectively;
   when the subscript $D$ is omitted, the domain of integration will be assumed to be all of $\Omega$.
\end{rem}

\subsection{Mesh-dependent DG norms}

In the analysis of \eqref{eq:ipdg}, it will be convenient to introduce the mesh-dependent DG norm $\iii \cdot \iii$ defined by
\begin{equation}
   \label{eq:dg-norm}
   \iii \bm v \iii^2
      = \| \nabla_h \bm v \|_0^2 + \| \alpha^{1/2} \llb \bm v \rrb \|_{0,\Gamma}^2.
\end{equation}
The IPDG bilinear form satisfies the following continuity and coercivity bounds.

\begin{lem}[{\cite[Lemma 2.4]{Antonietti2010}}]
   \label{lem:norm-equiv}
   For all $\bm u_h, \bm v_h \in \Vh$, it holds that
   \begin{gather}
      a(\bm u_h, \bm v_h) \lesssim \iii \bm u_h \iii \, \iii \bm v_h \iii, \\
      a(\bm u_h, \bm u_h) \gtrsim \iii \bm u_h \iii^2.
   \end{gather}
\end{lem}
Furthermore, the following useful $L^2$-norm estimates hold.
\begin{lem}[{\cite[Lemma 2.6]{Antonietti2010} and \cite[Lemma 3]{Pazner2021b}}]
   \label{lem:eigval-estimates}
   For all $\bm u_h \in \Vh$, it holds that
   \[
      \| \bm u_h \|_0^2 \lesssim a(\bm u_h, \bm u_h) \lesssim \eta \frac{p^4}{h^2} \| \bm u_h \|_0^2,
   \]
   and
   \[
      a(\bm u_h, \bm u_h) \lesssim \frac{p^4}{h^2} \| \bm u_h \|_0^2 + \eta \frac{p^2}{h} \| \llb \bm u_h \rrb \|_{0,\Gamma}^2.
   \]
\end{lem}

\subsection{Gauss--Lobatto nodal basis and degrees of freedom}
\label{sec:lobatto}

In this work, we consider a nodal basis for the reference space $\VV(\khat)$ (cf.~\eqref{eq:reference-space}) defined using the tensor-product of one-dimensional Lagrange interpolating polynomials defined at Gauss--Lobatto points.
For exposition, suppose $d = 2$, and so $\VV(\khat) = \QQ_{p,p-1} \times \QQ_{p-1,p}$.
The basis for $\QQ_{p,p-1}$ is given as follows.
Let $\{ \xi_i \}_0^p$ denote the set of $(p+1)$ Gauss--Lobatto points, and let $\{ \hat{\xi}_i \}_0^{p-1}$ denote the set of $p$ Gauss--Lobatto points.
For $0 \leq i \leq p$ and $0 \leq j \leq p-1$, let $\varphi_{ij}$ be the unique element of $\QQ_{p,p-1}$ satisfying
\[
   \varphi_{i} (\xi_\ell, \hat{\xi}_m) = \begin{cases}
      1, &\quad\text{if $i=\ell$, $j=m$}, \\
      0, &\quad\text{otherwise}.
   \end{cases}
\]
We can also write $\varphi_{ij}$ in tensor-product form,
\[
   \varphi_{ij} (\xi_\ell, \hat{\xi}_m)
   = \varphi_i(\xi_\ell) \hat{\varphi}_j(\hat{\xi}_m),
\]
where $\{ \varphi_i \}_{i=0}^p$ are the one-dimensional Lagrange interpolation polynomials associated with $\xi_i$, and $\{ \hat{\varphi}_i \}_{i=0}^{p-1}$ are those associated with $\hat{\xi}_i$.
An analogous construction is used to define the basis $\{ \psi_{ij} \}$ for $\QQ_{p-1,p}$.
The basis for $\VV(\khat)$ is formed by taking $2p(p+1)$ coordinate-aligned vector-valued multiples of these basis functions,
\[
   \{ \bm\varphi_{ij} : \bm\varphi_{ij} = (\varphi_{ij},0)^\tr \} \cup \{ \bm\psi_{ij} : \bm\psi_{ij} = (0,\psi_{ij})^\tr \}.
\]
The local basis for $\VV(\kk)$ for each $\kk \in \T$ is given using the Piola transformation \eqref{eq:piola}, and the basis for the global space $\Vh$ is obtained by the usual identification of shared degrees of freedom, ensuring continuity of the normal components.

This basis induces a dual set of degrees of freedom.
Let $\bm\theta_i$ denote the $i$th global basis function (corresponding to one of the $\bm\varphi_{jk}$ or $\bm\psi_{jk}$ local basis functions).
Then, $\bm x_i$ denotes the corresponding (physical) nodal point, and $\hat{\bm n}_i$ denotes the associated coordinate direction (either $(1,0)^\tr$ or $(0,1)^\tr$ depending on whether $\bm\theta_i$ corresponds to $\bm\varphi_{jk}$ or $\bm\psi_{jk}$, respectively).
Then, the corresponding degree of freedom is given by
\begin{equation}
   \label{eq:dof}
   \alpha_i(\bm u) = \hat{\bm n}_i^\tr \det(J(\bm x_i)) J^{-1}(\bm x_i) \bm u(\bm x_i).
\end{equation}
The associated nodal interpolation operator $\II_{\Vh}$ corresponds to pointwise interpolation at the nodal points and directions.

Consider an element $\kk \in \T$.
We will call a degree of freedom \textit{interior} if the associated Gauss--Lobatto node lies in the interior of $\kk$.
Let $\VV_0(\kk)$ denote the subset of $\VV(\kk)$ containing those functions for which all interior degrees of freedom vanish.
Then, the following trace and inverse trace inequalities hold.

\begin{lem}
   \label{lem:trace}
   For all $\bm v \in \VV(\kk)$,
   \[
      \| u \|_{0,\partial \kk}^2 \lesssim \frac{p^2}{h} \| u \|_{0,\kk}^2,
   \]
   and, for all $\bm w \in \VV_0(\kk)$,
   \[
      \| w \|_{0,\kk}^2 \lesssim \frac{h}{p^2} \| w \|_{0,\partial \kk}^2.
   \]
\end{lem}
\begin{proof}
   From \cite[Lemma 3.1]{Burman2007}, we recall the corresponding trace and inverse trace inequalities for $\QQ_p(\kk)$:
   for $u \in \QQ_p(\kk)$, it holds that
   \[
      \| u \|_{0,\partial \kk}^2 \lesssim \frac{p^2}{h} \| u \|_{0,\kk}^2,
   \]
   and, for $u \in \QQ_p^0(\kk)$, where $\QQ_p^0$ indicates the set of polynomials that vanish at the interior Gauss--Lobatto nodes, it holds that
   \[
      \| u \|_{0,\kk}^2 \lesssim \frac{h}{p^2} \| u \|_{0,\partial \kk}^2.
   \]
   The proof of \cite[Lemma 3.1]{Burman2007} extends immediately to functions in $\QQ_{p,p-1}$ and $\QQ_{p-1,p}$ and to vector-valued functions $\bm u \in \QQ_{p,p-1} \times \QQ_{p-1,p}$.

   Any $\bm v \in \VV(\kk)$ is given by $\bm v = \det(J)^{-1} J \hat{\bm v}$ where $\hat{\bm v} \in \QQ_{p,p-1} \times \QQ_{p-1,p}$, and so pointwise $\| \bm v(\bm x) \|_{\ell^2} \approx \| \hat{\bm v}(\bm x) \|_{\ell^2}$, up to constants depending on maximum and minimum values of $\det(J)$ and the $\ell^2$ condition number of the Jacobian matrix $J$, from which the trace inequality follows.
   Similarly, if $\bm v \in \VV_0(\kk)$, then the associated $\hat{\bm v}$ is an element of $\QQ_{p,p-1}^0 \times \QQ_{p-1,p}^0$, and so the inverse trace inequality holds.
\end{proof}

We now introduce the Oswald averaging operator $Q_h : \Vh \to \Vh$ (cf.~\cite{Oswald1993} and e.g.~\cite{Burman2007}).
For any $\bm x \in \Omega$, let $\mathcal{E}(\bm x)$ denote the set of elements $\kk \in \T$ containing $\bm x$.
Then, given $\bm v_h$, define its average value $\overline{\bm v_h}$ by
\[
   \overline{\bm v_h}(\bm x) = \frac{1}{\# \mathcal{E}(\bm x)}
      \sum_{\kk \in \mathcal{E}(\bm x)} \bm v_h|_{\kk}(\bm x).
\]
The operator $Q_h$ is defined by
\begin{equation}
   \label{eq:oswald}
   Q_h(\bm v_h) = \II_{\Vh} ( \overline{\bm v_h}).
\end{equation}
We remark that if all mesh elements $\kk$ are affine, then $Q_h(\bm v_h) \in [H^1(\Omega)]^d$.
This does not generally hold in the case of non-affine elements.
Note that by definition of $Q_h$, for $\bm w_h = \bm v_h - Q_h(\bm v_h)$, all interior degrees of freedom of $\bm w_h$ vanish.
This leads to the following estimate of $\bm w_h$ in terms of the jumps of $\bm v_h$.

\begin{lem}
   \label{lem:oswald}
   For all $\bm v_h \in \Vh$, it holds that
   \[
      \| \bm v_h - Q_h(\bm v_h) \|_0^2 \lesssim \frac{h}{p^2} \| \llb \bm v_h \rrb \|_{0,\Gamma}^2.
   \]
\end{lem}
\begin{proof}
   The proof follows that of \cite[Lemma 3.2]{Burman2007}, making use of the result of \Cref{lem:trace}.
\end{proof}

A key property about interpolation at Gauss--Lobatto points is that it is stability in the $H^1$ norm.
This result was proven in \cite{Pavarino1992,Pavarino1993}; a shorter proof with explicit bounds on the constants is provided in \Cref{sec:technical}.
We state this result presently.

\begin{lem}[{Cf.~\cite[Lemmas 2 and 4]{Pavarino1993}}]
   \label{lem:h1-interp-stability}
   Consider the space $\QQ_p^d$ (resp.~$\QQ_{p-1}^d$) of $d$-variate polynomials of degree at most $p$ (resp.~$p-1$) in each variable.
   Let $\II_{p-1} : \QQ_p \to \QQ_{p-1}$ denote the interpolation operator at the Gauss--Lobatto points, and let $\II_{p-1}^\oo{d} : \QQ_p^d \to \QQ_{p-1}^d$ denote its $d$-dimensional tensor product,
   \[
      \II_{p-1}^\oo{d} = \II_{p-1} \otimes \II_{p-1} \otimes \cdots \otimes \II_{p-1}.
   \]
   Then, for all $u \in \QQ_p^d$, it holds that
   \begin{align*}
      \| \II_{p-1}^\oo{d} u \|_0^2 &\leq \left( 2 + \frac{4}{2p-3} \right)^d \| u \|_0^2, \\
      | \II_{p-1}^\oo{d} u |_1^2 &\leq \left( 2 + \frac{4}{2p-3} \right)^{d-1} | u |_1^2.
   \end{align*}
\end{lem}

\noindent
Furthermore, interpolation at Gauss--Lobatto nodes satisfies the following error estimate.

\begin{lem}
   \label{lem:p-pm1-interp}
   For any $u \in \QQ_p$ it holds that
   \[
      \| u - \II_{p-1} u \|_0^2 \leq \frac{2}{4p^2 - 4p - 3} | u |_{1}^2.
   \]
\end{lem}

\noindent
The proof of this result is also deferred until \Cref{sec:technical}.

\section{Preconditioning}
\label{sec:preconditioning}

In this section, we study three preconditioners for the discretized problem \eqref{eq:variational-problem}.
The analysis and exposition will be restricted to the two-dimensional case;
however, the extension to $d = 3$ should follow analogously.
In \Cref{sec:subspace}, we consider a subspace correction preconditioner based on a lowest-order $H^1$-conforming coarse space together with overlapping vertex patch spaces.
In \Cref{sec:fictitious}, a simple fictitious space preconditioner is studied, using the degree-$p$ discontinuous Galerkin space as a fictitious space.
In \Cref{sec:auxiliary}, the fictitious space preconditioner is further developed, using the degree-$(p-1)$ discontinuous Galerkin space as an auxiliary space, combined with a block Jacobi smoother with blocks corresponding to the geometric entities in the mesh.
Convergence estimates of each of the preconditioners are provided.

\subsection{Subspace correction preconditioners}
\label{sec:subspace}

The strategy pursued in this section is that of subspace correction (i.e.\ domain decomposition and Schwarz methods, see \cite{Toselli2005,Xu1992,Xu2001}, among others).
For simplicity, we consider here additive methods (parallel subspace correction methods).
The goal is to reduce the problem \eqref{eq:variational-problem} posed in the $\Hdiv$ finite element space to a vector Poisson problem posed on a related space;
the related space is chosen so that the resulting problem is a canonical discretization of the vector Poisson problem, such that standard, efficient solvers exist, and may be used without modification.
The reduction to such an associated problem is performed by solving a collection of small, local problems, essentially playing the role of a smoother in a two-grid method;
the related canonical problem represents the coarse grid.

The finite element space $\Vh$ is decomposed into a sum of $K+1$ subspaces,
\begin{equation}
   \label{eq:space-decomp}
   \Vh = \VV_0 + \sum_{i=1}^K \VV_i,
\end{equation}
where by convention $\VV_0$ is a coarse global space, and $\VV_i$ for $i \geq 0$ are small, local spaces.
We define the coarse space $\VV_0$ to be the lowest-order vector-valued $H^1$-conforming finite element space.
As in the method of \Citeauthor{Pavarino1992} (cf.~\cite{Pavarino1992,Pavarino1993}), the local spaces $\VV_i$ are chosen to correspond to vertex patches defined as follows.
Let $\mathcal{V}(i)$ denote the set of elements $\kk \in \T$ containing the $i$th vertex;
let $\Omega_i = \bigcup_{\kk \in \mathcal{V}(i)} \kk$.
Define the space $\VV_i$ to be the subspace of $\Vh$ consisting of all functions $\bm v_h$ that vanish outside of $\Omega_i$.

The following simple result verifies that the coarse space $\VV_0$ is a subspace of $\Vh$.
\begin{lem}
   \label{lem:bilinear-subspace}
   If the element $\kk$ is straight-sided and $p \geq 2$, then the local space $\VV(\kk) = \det(J)^{-1} J \VV(\hat{\kk})$ contains the space of vector-valued bilinear polynomials $\QQ_1 \times \QQ_1$.
\end{lem}
\begin{proof}
   Since $\kk$ is straight-sided, the transformation $T_\kk$ is bilinear, and its Jacobian matrix takes the form
   \[
      J = \begin{pmatrix}
         a_{11} + b_{11} y & a_{12} + b_{12} x \\
         a_{21} + b_{21} y & a_{22} + b_{22} x
      \end{pmatrix}.
   \]
   Its adjugate is given by
   \[
      \operatorname{adj}(J) = \det(J) J^{-1}
      = \begin{pmatrix}
         a_{22} + b_{22} x & -a_{12} - b_{12} x \\
         -a_{21} - b_{21} y & a_{11} + b_{11} y
      \end{pmatrix},
   \]
   and so, for any $\bm u \in \QQ_1 \times \QQ_1$, $\bm v = \operatorname{adj}(J) \bm u \in \QQ_{2,1} \times \QQ_{1,2}$.
   Then, $\bm u = \det(J)^{-1} J \bm v$, and hence $\bm u \in \VV(\kk)$.
\end{proof}

For each space $\VV_i$, denote the elliptic projection $P_i : \Vh \to \VV_i$ by
\[
   a(P_i \bm u, \bm u_i) = a(\bm u, \bm u_i) \quad\text{for all $\bm u_i \in \VV_i$}.
\]
Let $A : \Vh \to \Vh$ denote the self-adjoint positive-definite operator such that $a(\bm u_h, \bm v_h) = (A \bm u_h, \bm v_h)$ for all $\bm u_h, \bm v_h \in \Vh$, where $(\cdot, \cdot)$ denotes the $L^2$ inner product.
Let $A_i$ denote the restriction of $A$ to the subspace $\VV_i$, and let $Q_i$ denote the $L^2$ projection onto $\VV_i$.
The preconditioned system $P$ (assuming exact solvers) is given by
\[
   P = \sum_{i=0}^K P_i = \left( \sum_{i=0}^K A_i^{-1} Q_i \right) A,
\]
where the last equality follows from the identity $A_i P_i = Q_i A$.
In practice, it may not be feasible to form the preconditioned system $P$ because it requires the computation of the inverses $A_i$, which (particularly for the coarse space $\VV_0$) may be prohibitively large.
Therefore, subspace inverses $A_i^{-1}$ can be replaced with spectrally equivalent approximate inverses $R_i$ to obtain the additive subspace preconditioner $B$
\[
   B = \sum_{i=0}^K R_i Q_i.
\]
The preconditioned system using approximate inverses is given by
\[
   T = B A = \sum_{i=0}^K T_i = \sum_{i=0}^K R_i Q_i A.
\]
The following well-known identity will be the main tool used to estimate the condition number of $T$.

\begin{lem}[{\cite[Lemma 2.4]{Xu2002}}]
   For any $\bm v_h \in \Vh$ it holds that
   \begin{equation}
      a(T^{-1} \bm v_h, \bm v_h) = \inf_{\substack{\bm v_i \in \VV_i \\ \sum \bm v_i = \bm v_h}} \sum_{i=0}^K a(T_i^{-1} \bm v_i, \bm v_i).
   \end{equation}
   In the case that $R_i$ is spectrally equivalent to $A_i^{-1}$, i.e.~$a(R_i \bm v_i, \bm v_i) \approx (\bm v_i, \bm v_i)$ it then follows that
   \begin{equation}
      \label{eq:inf-identity}
      a(T^{-1} \bm v_h, \bm v_h) \approx \inf_{\substack{\bm v_i \in \VV_i \\ \sum \bm v_i = \bm v_h}} \sum_{i=0}^K a(\bm v_i, \bm v_i).
   \end{equation}
\end{lem}

To establish an upper bound on the spectrum of $T$, we make use of a standard argument involving the finite overlap of the space decomposition \eqref{eq:space-decomp}.

\begin{lem}
   \label{lem:finite-overlap}
   For any $\bm v_h \in \Vh$, it holds that
   \begin{equation}
      a(T^{-1} \bm v_h, \bm v_h) \gtrsim a(\bm v_h, \bm v_h).
   \end{equation}
\end{lem}
\begin{proof}
   For any decomposition of $\bm v_h \in \Vh$ as $\bm v_h = \sum_{i=0}^K \bm v_i$ with $\bm v_i \in \VV_i$, it holds
   \begin{align*}
      a(\bm v_h, \bm v_h)
         = a\left(\sum \bm v_i, \sum \bm v_i\right)
         = \sum_{i=0}^{K} \Big( a(\bm v_i, \bm v_i) + \sum_{j \neq i} a(\bm v_i, \bm v_j) \Big).
   \end{align*}
   Note that for $i, j \geq 1$, then $a(\bm v_i, \bm v_j)$ is nonzero only when the spaces $\VV_i$ and $\VV_j$ correspond to vertices belonging to the same element.
   Therefore, the number of nonzero terms in the sum $\sum_{j \neq i} a(\bm v_i, \bm v_j)$ is bounded by the valence of the mesh $\T$ (plus one for the coarse space $\VV_0$).
   For each $i, j$, Cauchy--Schwarz and Young's inequality give $2 a(\bm v_i, \bm v_j) \leq a(\bm v_i, \bm v_i) + a(\bm v_j, \bm v_j)$, and so
   \[
      a(\bm v_h, \bm v_h) \lesssim \sum_{i=0}^K a(\bm v_i, \bm v_i).
   \]
   Since this holds for any such decomposition, \eqref{eq:inf-identity} implies
   \begin{align*}
      a(T^{-1} \bm v_h, \bm v_h)
         \approx \inf_{\substack{\bm v_i \in \VV_i \\ \sum \bm v_i = \bm v_h}} \sum_{i=0}^K a(\bm v_i, \bm v_i)
         \gtrsim a(\bm v_h, \bm v_h),
   \end{align*}
   providing a lower bound for the spectrum of $T^{-1}$.
\end{proof}

In order to establish a lower bound on the spectrum of $T$, it suffices to exhibit a \textit{stable decomposition} of $\bm v_h \in \Vh$.
We first construct the coarse approximation $\bm v_0 \in \VV_0$ to $\bm v_h$ using the Oswald averaging operator and nodal interpolation.
Recall that the Oswald averaging operator is defined by \eqref{eq:oswald}.
For any $\bm v_h \in \Vh$, its image $Q_h(\bm v_h)$ has well-defined values at all vertices in the mesh $\T$.
Let $\II_{\VV_0}$ denote lowest-order nodal interpolation at the mesh vertices.
Then, the coarse approximation $\bm v_0$ is given by
\[
   \bm v_0 = \II_{\VV_0} ( Q_h ( \bm v_h )).
\]
Let $\bm v_r$ denote the remainder $\bm v_r = \bm v_h - \bm v_0$.
For the $i$th mesh vertex, let $\bm \lambda_i$ denote the vector-value piecewise bilinear function that takes value 1 at vertex $i$, and vanishes at all other mesh vertices.
Note that the functions $\bm \lambda_i$ form a partition of unity, and so $\bm v_r = \sum_i \bm \lambda_i \odot \bm v_r$,
where $\odot$ denotes the elementwise (Hadamard) product.
Since the interpolation operator $\II_{\Vh}$ is a linear projection we have that $\bm v_r = \II_{\Vh} (\bm v_r) = \sum_i \II_{\Vh}(\bm \lambda_i \odot \bm v_r)$.
Therefore, defining $\bm v_i = \II_{\Vh}(\bm \lambda_i \odot \bm v_r)$ results in the decomposition
\begin{equation}
   \label{eq:decomp}
   \bm v_h = \bm v_0 + \sum_{i=1}^K \bm v_i, \quad \bm v_i \in \VV_i.
\end{equation}

\begin{lem}
   \label{lem:stable-decomp}
   The decomposition \eqref{eq:decomp} satisfies
   \[
      \sum_{i=0}^K \iii \bm v_i \iii^2 \lesssim \eta p^2 \iii \bm v_h \iii^2
   \]
   from which it follows that
   \[
      a(T^{-1} \bm v_h, \bm v_h) \lesssim \eta p^2 a(\bm v_h, \bm v_h).
   \]
   In the case that all mesh elements $\kk \in \T$ are affine, then the following improved bound is obtained,
   \[
      a(T^{-1} \bm v_h, \bm v_h) \lesssim a(\bm v_h, \bm v_h).
   \]
\end{lem}
\begin{proof}
   By \Cref{lem:oswald} and the eigenvalues estimates in \Cref{lem:eigval-estimates},
   \begin{equation}
      \label{eq:oswald-l2-dg}
      \begin{aligned}
         \iii \bm v_h - Q_h(\bm v_j) \iii^2
         \lesssim \frac{p^4}{h^2} \| \bm v_h - Q_h(\bm v_j) \|_0^2
         \lesssim \frac{p^2}{h} \| \llb \bm v_h \rrb \|_{0,\Gamma}^2
         \lesssim \iii \bm v_h \iii^2.
      \end{aligned}
   \end{equation}
   From this it follows that $\iii Q_h(\bm v_j) \iii^2 \lesssim \iii \bm v_h \iii^2$.
   Since $\bm v_0 = \II_{\VV_0}(Q_h(\bm v_j)) \in [H^1(\Omega)]^2$, $\llb \bm v_0 \rrb \equiv 0$.
   By stability of the nodal interpolant $\II_{\VV_0}$ in the $H^1$ norm, it follows that $\iii \bm v_0 \iii^2 \lesssim \iii \bm v_h \iii^2$.

   We now turn to $\bm v_i = \II_{\Vh}(\bm\lambda_i \odot \bm v_r) \in \VV_i$, where $\bm v_r = \bm v_h - \bm v_0$.
   Note that $\| \nabla_h \bm\lambda_i \odot \bm v_r \|_{0,\Omega_i}^2 \lesssim h^{-2} \| \bm v_r \|_{0,\Omega_i}^2 + \| \nabla_h \bm v_r \|_{0,\Omega_i}^2$.
   By stability of the interpolant (cf.~\cite{Pavarino1993} and \Cref{lem:h1-interp-stability}), it follows that $\| \nabla_h \bm v_i \|_{0,\Omega_i}^2 \lesssim h^{-2} \| \bm v_r \|_{0,\Omega_i}^2 + \| \nabla_h \bm v_r \|_{0,\Omega_i}^2 \lesssim \iii \bm v_h \iii^2$, using again \eqref{eq:oswald-l2-dg}.

   The term $\| \llb \bm v_i \rrb \|_{0,\Gamma_i}^2$ (where $\Gamma_i$ denotes the skeleton of $\Omega_i$) is bounded by
   \[
      \| \llb \bm v_i \rrb \|_{0,\Gamma_i}^2
      \lesssim \| \llb \bm v_r \rrb \|_{0,\Gamma_i}^2 + \| \II_{\Vh}(\bm\lambda_i \odot \bm v_r) - \bm\lambda_i \odot \bm v_r \|_{0,\Gamma_i}^2.
   \]
   By \Cref{lem:p-pm1-interp}, a scaling argument, and \Cref{lem:trace},
   \begin{align*}
      \| \II_{\Vh}(\bm\lambda_i \odot \bm v_r) - \bm\lambda_i \odot \bm v_r \|_{0,\Gamma_i}^2
         &\lesssim \frac{h^2}{p^2} \| \nabla_h \bm v_r \|_{0,\Gamma_i} \\
         &\lesssim h \| \nabla_h \bm v_r \|_{0,\Omega_i},
   \end{align*}
   and so $\eta p^2 h^{-1} \| \llb \bm v_i \rrb \|_{0,\Gamma_i}^2 \lesssim \eta p^2 \| \nabla_h \bm v_r \|_{0,\Omega_i}$,
   from which it follows that
   \[
      \iii \bm v_i \iii_{\Omega_i}^2 \lesssim \eta p^2 \iii \bm v_r \iii_{\Omega_i}^2 \lesssim \eta p^2 \iii \bm v_h \iii_{\Omega_i}^2.
   \]
   The conclusion then follows using the finite overlap of the subdomains $\Omega_i$.

   In the case that the mesh elements are all affine, then the element Jacobians are constant, from which the estimate $\| \llb \II_{\Vh}(\bm\lambda_i \odot \bm v_r) \rrb \|_{0,\Gamma_i}^2 \lesssim \| \llb \bm v_r \rrb \|_{0,\Gamma_i}^2$ follows, resulting in $\iii \bm v_i \iii_{\Omega_i}^2 \lesssim \iii \bm v_r \iii_{0,\Omega_i}^2$, which leads to the improved bounds.
\end{proof}

\begin{thm}
   \label{thm:subspace-conditioning}
   The condition number $\kappa(BA)$ of the preconditioned system $T = BA$ is independent of $h$.
   In the case that all mesh elements $\kk \in \T$ are affine, $\kappa(BA)$ is also independent of $p$ and $\eta$.
\end{thm}
\begin{proof}
   This follows immediately from \Cref{lem:finite-overlap} and \Cref{lem:stable-decomp}.
\end{proof}

\begin{rem}
   \label{rem:subspace-p-scaling}
   On meshes with non-affine meshes, \Cref{thm:subspace-conditioning} and \Cref{lem:stable-decomp} result in the bound $\kappa(BA) \lesssim \eta p^2$.
   The numerical results of \Cref{sec:results} indicate that this bound is pessimistic, and in practice, the condition number of the preconditioned system scales more favorably with the polynomial degree $p$.
\end{rem}

\begin{rem}[Implementation of the subspace correction preconditioner]
   As discussed previously, the inverses $A_i^{-1}$ of the subspace problems can be replaced with spectrally equivalent approximations, $R_i$.
   In the case of the local patch subspaces, the problem sizes do not grow with mesh refinement, and typically a direct solver can be used.
   Additionally, these problems are decoupled, and may be solved independently in parallel.
   With respect to the polynomial degree, the local problem size scales as $p^{2d}$ in $d$ dimensions, and so for very high-order problems the local patch solvers can represent a bottleneck.
   The coarse space $\VV_0$ is a global space, and on refined meshes the coarse problem $A_0$ can be quite large.
   However, $A_0$ is a standard $H^1$-conforming discretization of the vector Laplacian, for which a wide range of scalable and efficient solvers and preconditioners exist, any of which can be used for $R_0$.
   In \Cref{sec:results}, $R_0$ is taken to be one V-cycle of an algebraic multigrid method.
\end{rem}

\subsection{Fictitious space preconditioning}
\label{sec:fictitious}

In this section, we pursue the approach of fictitious space preconditioning (cf.~Nepomnyaschikh, \cite{Nepomnyaschikh1991,Nepomnyaschikh2007}).
This method allows us to dispense with the vertex patch solvers, which become computationally expensive for larger polynomial degree $p$.

As in the previous section, $A : \Vh \to \Vh$ is the symmetric positive-definite operator corresponding to the bilinear form $a(\cdot\,,\cdot)$.
We consider another space (in our case a finite-dimensional vector space; more generally, a Hilbert space) $\widetilde{\VV}$, termed the \textit{fictitious space}.
Let $R : \widetilde{\VV} \to \Vh$ be a surjective mapping, and let $\widetilde{A} : \widetilde{\VV} \to \widetilde{\VV}$ be another symmetric positive-definite operator.
Then, $B = R \widetilde{A}^{-1} R^\tr$ is a symmetric positive-definite operator that will serve as the fictitious space preconditioner.
The theory of Nepomnyaschikh gives the following spectral bounds.
\begin{lem}[{\cite[Lemma 2.3]{Nepomnyaschikh1991}}]
   \label{lem:fictitious}
   Suppose that, for all $\widetilde{v} \in \widetilde{\VV}$,
   \begin{equation}
      \| R \widetilde{v} \|_A^2 \leq c_R \| \widetilde{v} \|_{\widetilde{A}}^2,
   \end{equation}
   and, given $v \in \Vh$, there exists some $\widetilde{v} \in \widetilde{\VV}$ such that $v = R \widetilde{v}$, and
   \begin{equation}
      \| \widetilde{v} \|_{\widetilde{A}}^2 \leq c_S \| v \|_A^2,
   \end{equation}
   for some constants $c_R$ and $c_S$.
   Then, for $B = R \widetilde{A}^{-1} R^\tr$,
   \[
      c_S^{-1} (A^{-1} v, v) \leq (Bv, v) \leq c_R (A^{-1} v, v),
   \]
   and so the condition number of $BA$ is bounded by
   \[
      \kappa(BA) \leq c_R c_S.
   \]
\end{lem}

Applying this theory to the problem at hand, we aim to define the fictitious space $\widetilde{\VV}$, operator $\widetilde{A}$, and mapping $R$ such that the constants $c_R$ and $c_S$ are independent of the discretization parameters.
The fictitious space $\widetilde{\VV} = \Wh$ is given by the vector-valued discontinuous Galerkin space with polynomial degree $p$ defined on the mesh $\T$,
\[
   \Wh = \{ \bm w_h \in [L^2(\Omega)]^d : \bm w_h|_{\kk} \circ T_\kk \in [\mathcal{Q}_p]^d \}.
\]
The operator $\widetilde{A}$ is given by the interior penalty bilinear form \eqref{eq:ipdg} on the space $\Wh$.
The mapping $R : \Wh \to \Vh$ is defined as follows.
We first introduce the ``broken $\Hdiv$'' space $\Vhat$,
\[
   \Vhat = \{
      \vv_h \in [L^2(\Omega)]^d : \vv_h|_\kk \circ T_\kk \in \VV(\kk)
   \}.
\]
This spaces differs from $\Vh$ in that no continuity is enforced between elements.
A basis for $\Vhat$ can be constructed as in \Cref{sec:lobatto}, but no identification of shared global degrees of freedom is required.
In this way, each degree of freedom of $\Vh$ corresponds to a set of degrees of freedom of $\Vhat$, all belonging to the same geometric entity (vertex, edge, or element interior).
With this choice of basis, the injection $P : \Vh \hookrightarrow \Vhat$ can be represented as a boolean matrix with exactly one nonzero per row (but potentially multiple nonzeros per column).
The number of nonzeros in the $i$th column is the \textit{multiplicity} of the $i$th DOF, and is denoted $m_i$.
Let $M$ be the diagonal matrix with $M_{ii} = m_i$.
Then, define the operator $\Qhat : \Vhat \to \Vh$ by
\begin{equation}
   \label{eq:oswald-fictitious}
   \Qhat = M^{-1} P^\tr.
\end{equation}
Note that $\Qhat$ is a left-inverse of $P$, i.e.~$\Qhat P = I$.
under this mapping, the value of the $i$th degree of freedom is given by the average value of the associated degrees of freedom in the broken space.
This operator is similar to but distinct from that defined by \eqref{eq:oswald}

Let $\II_{\Vhat} : \Wh \to \Vhat$ be the nodal interpolant associated with the degrees of freedom given by \eqref{eq:dof},
and define $R : \Wh \to \Vh$ by
\begin{equation}
   \label{eq:transfer}
   R = \Qhat \II_{\Vhat}.
\end{equation}

\begin{lem}
   \label{lem:surjective}
   The transfer operator $R : \Wh \to \Vh$ defined by \eqref{eq:transfer} is surjective.
\end{lem}
\begin{proof}
   Surjectivity of the Oswald operator $\Qhat$ is immediate.
   It remains to show that nodal interpolation operator $\II : \Wh \to \Vhat$ is surjective.
   It suffices to consider a single element $\kk$.
   The associated mapping $T_\kk$ is bilinear.
   We show that for all $i$, there exists $\bm w_h \in \WW_h$ such that
   \begin{align*}
      \alpha_j(\II \bm w_h) &= \delta_{ij},
   \end{align*}
   where $\alpha_j$ denotes the $j$th degree of freedom functional defined by \eqref{eq:dof};
   this implies that $\II \bm w_h = \bm\theta_i$, from which surjectivity of $\II$ follows.
   Without loss of generality, assume $\hat{\bm n}_i = (1,0)^\tr$; the case of $(0,1)^\tr$ follows by the same reasoning.
   Let $\bm x_i = (\xi_\ell, \hat{\xi}_m)$ denote the associated Gauss--Lobatto nodal point.
   Choose $z \in \QQ_{p,p-1}$ to be the unique polynomial such that
   \[
      z(\xi_{\ell'}, \hat{\xi}_{m'}) = \delta_{\ell\ell'} \delta_{mm'} \frac{1}{\det J(\xi_\ell, \hat{\chi}_m)}
   \]
   and define $\bm w_h$ by
   \[
      \bm w_h(x, y) = J(x,y) \begin{pmatrix} z(x,y) \\ 0 \end{pmatrix}
      = \begin{pmatrix}
         J_{11}(x,y) z(x,y) \\
         J_{21}(x,y) z(x,y)
      \end{pmatrix}.
   \]
   Then,
   \begin{align*}
      \alpha_j ( \II \bm w_h )
         &= \hat{\bm n}_i^\tr \det(J(\bm x_i)) J^{-1} (\bm x_i) \bm w_h(\bm x_i) \\
         &= \hat{\bm n}_i^\tr \det(J(\bm x_i)) \begin{pmatrix} z(x,y) \\ 0 \end{pmatrix} \\
         &= \delta_{ij}.
   \end{align*}
   It remains to show that $\bm w_h$ belongs to the space $\Wh$.
   Since $T_\kk$ is bilinear, it follows that $J_{11}$ and $J_{21}$ are affine functions in $y$ only.
   Since $z(x,y)$ is of degree $(p-1)$ in $y$, it follows that $J_{11} z \in \QQ_{p,p}$ and $J_{21} z \in \QQ_{p,p}$, and thus $\bm w_h \in \Wh$.
\end{proof}

\begin{lem}
   \label{lem:R-projection}
   If the element transformations $T_\kk$ are affine, then $\Vh \subseteq \Wh$, and the operator $R$ is a projection.
\end{lem}
\begin{proof}
   If $T_\kk$ is affine, then the Jacobian matrix $J$ (and hence its determinant) are spatially constant, and the elements of $\VV(\kk)$ are polynomials.
   Since $\VV(\kk) = \det(J_\kk) J_\kk \QQ_{p,p-1} \times \QQ_{p-1,p}$, elements of $\VV(\kk)$ are of degree at most $p$ in each variable.
   Therefore, $\widehat{\VV}_h$ consists of vector-valued piecewise polynomial functions of degree at most $p$ in each variable (with no continuity enforced between elements), and so $\widehat{\VV}_h \subseteq \Wh$.
   Since $\Vh \subseteq \widehat{\VV}_h$, the first statement follows.

   To show that $R$ is a projection, we show that $R$ acts as the identity on the subspace $\Vh$.
   Note that $R = Q \II_{\Vhat}$, and $\II_{\Vhat}$ acts on the identity on $\widehat{\VV}_h \subseteq \Wh$, and $Q$ acts as identity on $\Vh \subseteq \widehat{\VV}_h$.
   Therefore, $R \bm v_h = \bm v_h$ for all $\bm v_h \in \Vh$, and $R^2 = R$.
\end{proof}

We now provide estimates for the constants $c_R$ and $c_S$ from \Cref{lem:fictitious}, resulting in bounds on the condition number of the preconditioned system.
For certain meshes, bounds independent of discretization parameters $h$, $p$, and $\eta$ can be proven;
in the case of general meshes, the bounds will be independent of mesh size $h$, but potentially dependent on polynomial degree and penalty parameter.
In light of \Cref{lem:norm-equiv}, it will be convenient to work with the mesh-dependent DG norm $\iii \cdot \iii$ defined by \eqref{eq:dg-norm}.

\begin{thm}
   \label{thm:fictitious}
   On general meshes, the constants $c_S$ and $c_R$ from \Cref{lem:fictitious} are independent of $h$, and so the condition number $\kappa(BA)$ is bounded independent of $h$.
   If the element transformations $T_\kk$ are affine, then $c_S = 1$.
   If the mesh is Cartesian (or affinely transformed Cartesian grid) then $c_R$ and $\kappa(BA)$ are also independent of $p$ and $\eta$.
\end{thm}
\begin{proof}
   We first show that the operator norms of $R$ and its right-inverse $S$ (defined according to the procedure of \Cref{lem:surjective}) are bounded independent of $h$.
   A scaling argument shows that for all $\vv_h \in \Vh$, $S \vv_j$ satisfies $\| \vv_h - S \vv_h \|_0^2 \leq C h^2 \| \nabla_h \vv_h \|_0^2$ and $\| \nabla_h (\vv_h - S \vv_h) \|_0^2 \leq C \| \nabla_h \vv_h \|_0^2$, where the constant $C$ is independent of $h$.
   From this, the trace inequality \Cref{lem:trace} implies that
   \begin{align*}
      \| \llb \vv_h - S \vv_h \rrb \|_{0,\Gamma}^2
         \leq \sum_{\kk \in \T} \| \vv_h - S \vv_h \|_{0,\partial\kk}^2
         \leq C \frac{p^2}{h} \sum_{\kk \in \T} \| \vv_h - S \vv_h \|_{0,\kk}^2
         \leq C p^2 h \sum_{\kk \in \T} \| \nabla_h \vv_h \|_{0,\kk}^2,
   \end{align*}
   from which we have $\iii \vv_h - S \vv_h \iii \leq C \iii \vv_h \iii$.
   Boundedness of $S$ in the $\iii\cdot\iii$ norm follows from the triangle inequality, implying that $c_S$ is bounded independent of $h$.
   If the mesh is affine, then from \Cref{lem:R-projection}, $S = I$, and $c_S = 1$.

   In order to bound $R$, we bound the operator norms of the nodal interpolant $\II_{\Vhat}$ and the Oswald operator $\Qhat$.
   By a straightforward analogue of \Cref{lem:h1-interp-stability}, it can be seen that for any $\bm w_h \in \Wh$,  $\| \nabla_h \II_{\Vhat} \bm w_h \|_0^2 \lesssim \| \nabla_h \bm w_h \|_0^2$.
   The interpolation operator further satisfies the approximation property $\| \bm w_h - \II_{\Vhat} \bm w_h \|_0^2 \leq C h^2 \| \nabla_h \bm w_h \|_0^2$.
   By the same argument as for the operator $S$, this implies that $\II_{\Vhat}$ is bounded in the DG norm $\iii \cdot \iii$ independent of $h$.

   To bound the operator $\Qhat$ in the norm $\iii \cdot \iii$ we briefly repeat an argument from \Cref{lem:stable-decomp}.
   Since the Oswald operator $\Qhat$ leaves interior degrees of freedom unchanged, $\bm v_h - \Qhat \bm v_h$ vanishes at all interior Gauss--Lobatto nodes, and hence \Cref{lem:oswald} applies.
   So, for all $\bm v_h \in \Vhat$, $\| \bm v_h - \Qhat(\bm v_h) \|_0^2 \lesssim \frac{h}{p^2} \| \llb \bm v_h \rrb \|_{0,\Gamma}^2$.
   Together with \Cref{lem:eigval-estimates} it follows that
   \begin{align*}
      \iii \bm v_h - \Qhat(\bm v_h) \iii^2
         \lesssim \eta \frac{p^4}{h^2} \| \bm v_h - \Qhat(\bm v_h) \|_0^2
         \lesssim \eta \frac{p^2}{h} \| \llb \bm v_h \rrb \|_{0,\Gamma}^2 \lesssim \iii \bm v_h \iii^2,
   \end{align*}
   and so from the triangle inequality it holds that $\iii \Qhat \bm v_h \iii^2 \lesssim \iii \bm v_h \iii^2$.
   Since $R = \Qhat \II_{\Vhat}$, the above bounds imply that $c_R$ is bounded independent of $h$.

   In the special case that the mesh is Cartesian, note that on any edge $e \in \Gamma$, the operator $\II_{\Vhat}$ acts as identity on the tangential component of $\bm w_h$, and degree-$(p-1)$ interpolation of the normal component.
   Linearity of the interpolant together with \Cref{lem:l2-interp-stable-1d} imply that
   \begin{equation}
      \| \llb \II_{\Vhat} (\bm w_h) \rrb \|_{0,e}^2
         = \| \II_{\Vhat} ( \llb (\bm w_h) \rrb ) \|_{0,e}^2
         \lesssim \| \llb \bm w_h \rrb \|_{0,e}^2
   \end{equation}
   Summing over all edges in the mesh skeleton and using the $H^1$-stability of the interpolant results in
   \[
      \iii \II_{\Vhat} \bm w_h \iii^2 \lesssim \iii \bm w_h \iii^2,
   \]
   and so in this case, $c_R$ is also bounded independent of $p$ and $\eta$.
\end{proof}

In practice, it is not feasible to construct the exact fictitious space preconditioner defined by $B = R \widetilde{A}^{-1} R^\tr$, since computing the inverse $\widetilde{A}^{-1}$ is prohibitive.
In fact, the fictitious space $\Wh$ is larger than the $\Hdiv$ space $\Vh$, and so the associated matrix $\widetilde{A}$ is larger $A$, rendering inversion via direct methods disadvantageous.
However, a large number of effective preconditioners for $\widetilde{A}$ have been described in the literature (see e.g.~\cite{Brix2014,Antonietti2016,Pazner2020a,Pazner2021b,Pazner2023});
among these are preconditioners with optimal complexity, with resulting condition number independent of discretization parameters.
Then, $\widetilde{A}^{-1}$ can be replaced by a spectrally equivalent preconditioner $\widetilde{B}$, and $B' = R \widetilde{B} R^\tr$ is used as a preconditioner for $A$.

\subsubsection{Matrix-free fictitious space preconditioning}
\label{sec:matrix-free}

For high-order discretizations, the memory requirements and computational cost associated with assembling the system matrix may be prohibitive;
the number of nonzeros in the matrix scales like $\mathcal{O}(p^{2d})$, and the number of operations required to assemble the matrix using traditional algorithms scales like $\mathcal{O}(p^{3d})$ (although sum factorization techniques can reduce the assembly cost to $\mathcal{O}(p^{2d+1})$, see e.g.~\cite{Melenk2001}).
However, the action of the operator can be computed with $\mathcal{O}(p^{d+1})$ operations and optimal $\mathcal{O}(p^d)$ memory requirements.
For this reason, matrix-free preconditioners, the action of which can be computed with access to the entries of the original matrix, are advantageous in the context of high-order discretizations.

The fictitious space preconditioner described in \Cref{sec:fictitious} lends itself to matrix-free implementation.
The action of the preconditioner $B' = R \widetilde{B} R^\tr$ given is computed as a sequence of three operator applications.
We show that the action of each of $R$, $\widetilde{B}$, and $R^\tr$ can be computed efficiently matrix-free.

The operator $R = \Qhat \II_{\Vhat}$ and its transpose can be computed without assembling their matrix representations.
The interpolation operator $\II_{\Vhat}$ is local to each element (block-diagonal).
Each block $\II_{\Vhat(\kk)}$ takes the form
\[
   \II_{\Vhat(\kk)} =
   \begin{pmatrix}
      J_{p} \otimes J_{p-1} & 0 \\
      0 & J_{p-1} \otimes J_p
   \end{pmatrix}
   \operatorname{diag} \left(
      \hat{\bm n}_i^\tr \det(J(\bm x_i)) J^{-1}(\bm x_i)
   \right).
\]
The action of the interpolation matrices $J_{p} \otimes J_{p-1}$ and $J_{p-1} \otimes J_p$ can be computed with explicitly forming the Kronecker product using standard sum factorization algorithms (cf.~\cite{Orszag1980}).
The Oswald operator $\Qhat : \Vhat \to \Vh$ and its transpose can be expressed as weighted gather and scatter operations.

It remains to express the action of $\widetilde{B}$ in a matrix-free way, without access to the assembled DG system matrix $\widetilde{A}$.
There are a number of strategies to achieve this;
presently, we take the approach of low-order-refined spectral equivalence using the method described in \cite{Pazner2023}.
In this method, we consider the lowest-order DG space $\bm W_0$ consisting of piecewise constant functions on a refined mesh $\T_0$, obtained by subdividing each element $\kk \in \T$ into $(p+1)^2$ subelements whose vertices are given by the Cartesian product of $p+2$ Gauss--Lobatto points.
Since the functions $\bm w_h \in \bm W_0$ are piecewise constant, the interior penalty bilinear form reduces simply to the penalty term,
\[
   a_0(\bm v_h, \bm w_h) = \langle \alpha_0 \llb \bm v_h \rrb, \llb \bm w_h \rrb  \rangle.
\]
Choosing $\alpha_0$ in terms of the Gauss--Lobatto quadrature points and weights as in \cite[Theorem 4.6]{Pazner2023} ensures that the resulting low-order-refined matrix $A_0$ is spectrally equivalent to the high-order DG discretization $\widetilde{A}$, independent of $h$, $p$, and penalty parameter $\eta$.
Unlike the high-order matrix $\widetilde{A}$, for which the number of nonzeros per row scales as $\mathcal{O}(p^d)$, the number of nonzeros per row of $A_0$ is bounded independent of $p$; the maximum number of nonzeros per row in $A_0$ is $2d + 1$.
Thus, the number of operations and storage required to construct $A_0$ is optimal ($\mathcal{O}(1)$ per degree of freedom).
Furthermore, $A_0$ possesses favorable properties that make it amenable to preconditioning using algebraic multigrid methods:
$A_0$ is an M-matrix and can be written as a weighted graph Laplacian; theory for algebraic multigrid convergence is well-developed in these cases \cite{Ruge1987}.
Additionally, it can be shown that classical AMG applied to $A_0$ will converge independent of the penalty parameter $\eta$ \cite{Pazner2023}.

\subsection{Auxiliary space preconditioning}
\label{sec:auxiliary}

We further develop the fictitious space preconditioner of the preceding section to construct an auxiliary space preconditioner (cf.~Xu, \cite{Xu1996}), based on the application of \Cref{lem:fictitious} to product spaces.
In this case, an auxiliary space $\VV_0$ with symmetric positive-definite operator $A_0 : \VV_0 \to \VV_0$ and transfer operator $\Pi : \VV_0 \to \Vh$ is used in conjunction with a smoother $D : \Vh \to \Vh$ to construct a preconditioner $B = D^{-1} + \Pi A_0^{-1} \Pi^\tr$.
The following result bounds the condition number of $BA$.

\begin{lem}[{\cite[Theorem 2.1]{Xu1996}}]
   Consider a space $\VV_0$ with symmetric positive-definite operator $A_0 : \VV_0 \to \VV_0$, transfer operator $\Pi : \VV_0 \to \Vh$, and symmetric positive-definite smoother $D : \Vh \to \Vh$ such that
   \[
      \begin{aligned}
         \| \Pi \vv_0 \|_A &\leq c_{\Pi} \| \vv_0 \|_{A_0} \qquad&&\text{for all $\vv_0 \in \VV_0$}, \\
         \| \vv \|_A &\leq c_D \| \vv \|_D \qquad&&\text{for all $\vv \in \Vh$},
      \end{aligned}
   \]
   and, suppose that for all $\vv \in \Vh$, there exists $\vv_0 \in \VV_0$ such that
   \[
      \| \vv_0 \|_{A_0}^2 + \| \vv - \Pi \vv_0 \|_D^2 \leq c_0^2 \| \vv \|_A^2.
   \]
   Then, the preconditioner $B = D^{-1} + \Pi A_0^{-1} \Pi^\tr$ satisfies
   \[
      \kappa(BA) \leq c_0^2 (c_D^2 + c_\Pi^2).
   \]
\end{lem}

This theorem can be proven by applying \Cref{lem:fictitious} with the fictitious space given by the product space $\widehat{\VV} = \VV_0 \times \Vh$.
Note that in the context of auxiliary space methods, the transfer operator $\Pi$ is not required to be surjective.

In this setting, we choose the auxiliary space to be the degree-$(p-1)$ DG space,
\[
   \WW_0 = \{ \bm w_h \in [L^2(\Omega)]^d : \bm w_h|_\kk \circ T_\kk \in \QQ_{p-1}^d \}.
\]
The operator $A_0$ is defined by the IPDG bilinear form \eqref{eq:ipdg} restricted to the space $\WW_0$ (note that $\WW_0 \subseteq \Wh$, where $\Wh$ is the fictitious space from the preceding section).
The transfer operator $\Pi$ is simply the restriction of $R$ defined by \eqref{eq:transfer} to the smaller space $\WW_0 \subseteq \Wh$.
The smoother $D$ is defined to be the block Jacobi (additive Schwarz) solver where each block consist of all nodes lying on a given geometric entity (vertex, edge interior, or element interior).

\begin{thm}
   \label{thm:auxiliary}
   Given the above definitions, the constants $c_{\Pi}$ and $c_D$ are independent of $h$, $\eta$, and $p$, and the constant $c_0$ is independent of $h$.
   The condition number $\kappa(BA)$ is independent of $h$.
\end{thm}
\begin{proof}
   Since $\QQ_{p-1}^d \subseteq \VV(\kk)$ for all $\kk$, $\Pi = \Qhat$ where $\Qhat$ is the Oswald operator defined by \eqref{eq:oswald-fictitious}.
   The proof of \Cref{thm:fictitious} gives a bound on the operator norm of $\Qhat$, implying that the constant $c_\Pi$ is independent of $h$, $p$, and $\eta$.

   The constant $c_D$ is uniformly bounded because of the finite overlap property;
   the $j$th block of the block Jacobi smoother $D$ corresponds to a subspace $\VV_j$ spanned by the basis functions belonging to a geometric entity (vertex, edge interior, or element interior).
   Given $\vv_i \in \VV_i$ and $\vv_j \in \VV_j$, then $a(\vv_i, \vv_j)$ is nonzero only when the associated geometric entities belong to the same element.
   The number of overlapping subspaces is bounded by a constant times the maximum vertex valence of the mesh.
   Note that $\Vh = \VV_1 \oplus \VV_2 \oplus \cdots \oplus \VV_M$, and so for any $\vv \in \Vh$, we can write $\vv = \vv_1 + \vv_2 + \cdots + \vv_M$ with $\vv_i \in \VV_i$, from which we have
   \begin{align*}
      \| \vv \|_A^2
         = a(\vv, \vv)
         = a\left(\sum_{i=1}^M \vv_i, \sum_{i=1}^M \vv_i\right)
         = \sum_{i,j=1}^M a(\vv_i, \vv_j)
         \lesssim \sum_{i}^M a(\vv_i, \vv_i)
         = \| \vv \|_D^2,
   \end{align*}
   using Young's inequality and the finite overlap property.

   To bound $c_0$, set $\vv_0 = \II_{\WW_0}(\vv)$.
   Stability of the interpolant $\II_{\WW_0}$ implies that $\| v_0 \|_{A_0}^2 \lesssim \| \vv \|_A^2$, independent of $h$, $p$, and $eta$.
   Note that $\| \vv - \vv_0 \|_0^2 \lesssim \frac{h^2}{p^2} \| \nabla_h \vv \|_0^2$ by \Cref{lem:p-pm1-interp} and a scaling argument.
   Let $\vv_D = \vv - \vv_0$ and expand $\vv_D = \sum_{i=1}^M \vv_i$.
   Then, by the $L^2$ norm estimates of \Cref{lem:eigval-estimates},
   \begin{align*}
      \| \vv_D \|_D^2
         &= \sum_{i=1}^M \| \vv_i \|_A^2
         \lesssim \eta \frac{p^4}{h^2} \sum_{i=1}^M \| \vv_i \|_0^2
         \lesssim \eta \frac{p^4}{h^2} \| \vv_D \|_0^2
         \lesssim \eta p^2 \| \vv \|_0^2
         \lesssim \eta p^2 \| \vv \|_A^2,
   \end{align*}
   proving that $c_0$ is independent of $h$.
\end{proof}

\begin{rem}
   Although the proof of \Cref{thm:auxiliary} results in the bound $c_0 \lesssim \eta p^2$,
   numerical results indicate that this bound is pessimistic, and that $c_0$ scales only mildly with $p$.
   This is similar to the case of the subspace correction preconditioner as discussed in \Cref{rem:subspace-p-scaling}.
   Heuristically, performance similar to $p$-multigrid with Jacobi or block-Jacobi smoothing can be expected (see, for example, \cite{Sundar2015}).
   On affine meshes, the constant $c_0$ is also expected to be independent of $\eta$.
   The performance of the auxiliary space preconditioner $B$ is studied numerically in \Cref{sec:results}.
\end{rem}

\subsection{Application to the Stokes system}
\label{sec:stokes}

Consider the Stokes problem
\begin{equation}
   \label{eq:stokes}
   \left\{~
   \begin{aligned}
      -\Delta \bm u + \nabla p &= f \quad&& \text{in $\Omega$,}\\
      \nabla \cdot \bm u &= 0 \quad&& \text{in $\Omega$,}\\
      \bm u &= \bm u_D \quad&& \text{on $\partial\Omega$,}
   \end{aligned}\right.
\end{equation}
where $\bm u$ is the velocity field and $p$ is the pressure.
This problem can be discretized using the $\Hdiv$-conforming discontinuous Galerkin method as follows:
find $(\bm u_h, p_h) \in \Vh \times W_h$ such that
\begin{align*}
   a(\bm u_h, \bm v_h) + b(\bm w_h, p_h) &= (f, \bm w_h), \\
   b(\bm u_h, q_h) &= 0,
\end{align*}
for all $(\bm v_h, q_h) \in \Vh \times W_h$, where the pressure space $W_h$ is the discontinuous Galerkin space of degree $p$, $a(\cdot\,,\cdot)$ is the interior penalty bilinear form defined by \eqref{eq:ipdg}, and $b(\bm u_h, q_h) = -(\nabla\cdot\bm u_h, q_h)$.
This method was introduced by \citeauthor{Cockburn2006} in \cite{Cockburn2006}.
This discretization results in exactly divergence-free velocity fields, and as such possesses favorable features such as pressure robustness \cite{John2017}.

The linear system associated with the discretization \eqref{eq:stokes} takes the form of the saddle-point system
\begin{equation}
   \label{eq:saddle-point}
   \underbrace{
   \begin{pmatrix}
      A & D^\tr \\
      D & 0
   \end{pmatrix}}_{\mathcal{A}}
   \begin{pmatrix} \bm{\mathsf{u}} \\ \mathsf{p} \end{pmatrix}
   = \begin{pmatrix} \bm{\mathsf{f}} \\ \mathsf{0} \end{pmatrix}.
\end{equation}
We presently consider the standard approach of block-diagonal preconditioning for this system \cite{Benzi2005}.
The ideal block-diagonal preconditioner is given by
\begin{equation}
   \label{eq:block-diag}
   \mathcal{B} = \begin{pmatrix}
      A^{-1} & 0 \\
      0 & S^{-1}
   \end{pmatrix},
\end{equation}
where $S = D A^{-1} D^\tr$ is the (negative) pressure Schur complement.
The diagonal blocks in \eqref{eq:block-diag} are replaced with spectrally equivalent approximations;
$A^{-1}$ can be replaced with any of the proposed preconditioners from the preceding sections.
The Schur complement $S$ is spectrally equivalent to the DG mass matrix $M$, and so $S^{-1}$ may be replaced either with $M^{-1}$ or, for example, $\widetilde{M}^{-1}$, where $\widetilde{M} = \operatorname{diag}(M)$;
the DG mass matrix using the Gauss--Lobatto nodal basis is spectrally equivalent to its diagonal, independent of $h$ and $p$ \cite{Teukolsky2015}.
By replacing the blocks with approximations, we obtain a preconditioner $\widetilde{\mathcal{B}}$.

The preconditioned Stokes system $\mathcal{B} \mathcal{A}$ is uniformly well-conditioned independent of discretization parameters.
However, the constants of approximation in the spectrally equivalent Schur complement approximation $S \approx M$ depend on the inf--sup constant, which in turn depends on the DG penalty parameter $\eta$.
So, even in the case where $A^{-1}$ can be approximated independent of $\eta$, the resulting block-diagonal preconditioned system will have $\eta$-dependent condition number.
This inf--sup dependence can be avoided using divergence-free multigrid methods as in \cite{Kanschat2015}; this approach is not pursued in the present work.

\section{Gauss--Lobatto interpolation estimates}
\label{sec:technical}

In this section, we provide some results on polynomial interpolation at the Gauss--Lobatto points;
these are required to complete the deferred proofs for the technical lemmas from \Cref{sec:lobatto}.
Many of these results are related to those of \Citeauthor{Bernardi1992} \cite{Bernardi1992}.
We will use properties of the orthogonal Legendre polynomials $P_n$ satisfying
\begin{equation}
   \label{eq:legendre-l2}
   \int_{-1}^1 P_m(x) P_n(x) \, dx = \frac{2}{2n + 1} \delta_{mn}.
\end{equation}
The $n$ Gauss--Lobatto nodes $x_i$ are given as the zeros of the degree-$n$ polynomial
\begin{equation}
   \label{eq:lobatto-chi}
   \chi_n(x) = P_{n-1}'(x) (x^2 - 1).
\end{equation}
In \cite[Section 6]{Bernardi1992}, $L^2$ stability of degree-$(p-1)$ Gauss--Lobatto interpolation of polynomials of degree $p$ was shown;
the following lemma provides explicit constants for the stability bounds.
\begin{lem}
   \label{lem:l2-interp-stable-1d}
   Let $\II_{p-1} : C^0 \to \QQ_{p-1}$ denote nodal interpolation at the $p$ Gauss--Lobatto nodes.
   For $u \in \QQ_p$, it holds that
   \[
      \| \II_{p-1} u \|_0^2 \leq \left( 2 + \frac{4}{2p-3} \right) \| u \|_0^2
   \]
   and this estimate is sharp.
\end{lem}

\begin{proof}
   Let $\chi = \chi_p$ denote the degree-$p$ polynomial that vanishes at the $p$ Gauss--Lobatto points.
   Since $\QQ_p = \operatorname{span} \{ \QQ_{p-1}, \chi \}$, we can write any $u \in \QQ_p$ as $u = u_{p-1} + c \chi$, for some coefficient $c$, where $\deg(u_{p-1}) = p - 1$, and $\chi$ is given by \eqref{eq:lobatto-chi}.
   Since $\chi$ vanishes at the Gauss--Lobatto points, $\II_{p-1} u = u_{p-1}$, and $\| \II_{p-1} u \|_0 = \| u_{p-1} \|_0$.
   Furthermore,
   \begin{equation}
      \label{eq:u-chi}
      \| u \|_0^2
         = \int_{-1}^1 (u_{p-1} + c \chi)^2 \, dx
         = \| u_{p-1} \|_0^2 + \| c \chi \|_0^2 + 2 c \int_{-1}^1 u_{p-1} \xi \, dx.
   \end{equation}
   We bound the third term on the right-hand side.
   Bonnet's recursion formulas
   \begin{align*}
      P_{p-1}'(x) (x^2 - 1) &= (p-1) \left( x P_{p-1}(x) - P_{p-2}(x) \right), \\
      (2p - 1) x P_{p-1}(x) &= p P_p(x) + (p-1) P_{p-2}(x),
   \end{align*}
   together result in
   \begin{equation}
      \label{eq:chi}
      \chi(x) = \frac{p^2 - p}{2p-1} \left( P_p(x) - P_{p-2}(x) \right),
   \end{equation}
   which then gives by \eqref{eq:legendre-l2}
   \begin{equation}
      \label{eq:chi-norm}
      \| \chi \|_0^2 = \left(
         \frac{p^2 - p}{2p - 1}
      \right)^2 \left(
         \frac{2}{2p+1} + \frac{2}{2p-3}
      \right).
   \end{equation}
   Expanding $u_{p-1}$ in terms of the Legendre polynomials
   \begin{equation}
      \label{eq:leg-expansion}
      u_{p-1}(x) = \sum_{i=0}^{p-1} a_i P_i(x).
   \end{equation}
   for some coefficients $a_i$ results in
   \begin{equation}
      \label{eq:upm1-norm}
      \| u_{p-1} \|_0^2 = \sum_{i=0}^{p-1} a_i^2 \frac{2}{2i - 1}.
   \end{equation}
   Using \eqref{eq:chi} and \eqref{eq:leg-expansion}, we compute
   \begin{align*}
      \int_{-1}^1 u_{p-1} \chi \, dx
         &= - \frac{p^2 - p}{2p-1} \int_{-1}^1 a_{p-2} P_{p-2}(x)^2 \, dx \\
         &= - a_{p-2} \frac{p^2 - p}{2p-1} \frac{2}{2p - 3} \\
         &\leq \gamma \| u_{p-1} \|_0 \| \chi \|_0
   \end{align*}
   by comparing with \eqref{eq:chi-norm} and \eqref{eq:upm1-norm}, where
   \[
      \gamma = \sqrt{\frac{2}{2p-3}} \Bigg / \sqrt{\frac{2}{2p+1} + \frac{2}{2p-3}}.
   \]
   From \eqref{eq:u-chi} it then holds
   \begin{align*}
      \| u \|_0^2
      &\geq \| u_{p-1} \|_0^2 + \| c \chi \|_0^2 - 2 \left| c \int_{-1}^1 u_{p-1} \chi \, dx \right| \\
      &\geq \| u_{p-1} \|_0^2 + \| c \chi \|_0^2 - 2 |c| \gamma \|u_{p-1} \|_0 \| \chi \|_0 \\
      &\geq (1 - \gamma^2) \| u_{p-1} \|_0^2,
   \end{align*}
   where the last step follows from Young's inequality.
   This implies
   \[
      \| J_{p-1} u \|_0^2 \leq \frac{1}{1 - \gamma^2} \| u \|_0^2,
   \]
   and the conclusion follows.
   Sharpness of the inequality can be seen by choosing $u$ to be the properly scaled linear combination of $P_{p-2}$ and $\chi$.
\end{proof}

The following expression for the $H^1$ seminorm of $\chi$ will also be useful
\begin{lem}
   \label{lem:chi-h1-seminorm}
   For $\chi = \chi_p$ defined by \eqref{eq:lobatto-chi}, it holds that
   \[
      | \chi |_1^2 = \frac{2(p^2-p)^2}{2p - 1}.
   \]
\end{lem}
\begin{proof}
   Recalling from the proof of \Cref{lem:l2-interp-stable-1d}, we have
   \[
      \chi(x) = \frac{p^2 - p}{2p-1} \left( P_p(x) - P_{p-2}(x) \right),
   \]
   which, combined with the well-known identity
   \[
      (2n + 1) P_n(x) = \frac{d}{dx} \left( P_{n+1}(x) - P_{n-1}(x) \right)
   \]
   gives
   \begin{equation}
      \label{eq:chi-derivative}
      \chi'(x) = (p^2 - p) P_{p-1}(x).
   \end{equation}
   The result then follows from \eqref{eq:legendre-l2}.
\end{proof}

Note also that the operator $\II_{p-1} : H^1 \to \QQ_{p-1}$ is stable with respect to the $H^1$ seminorm, i.e.\ $| \II_{p-1} u |_1 \lesssim | u |_1$ for all $u \in H^1$, with constant independent of $p$ (see \cite[Corollary 4.6]{Bernardi1992}).
Restricting the operator to $\QQ_p$, the implied constant in the inequality is exactly 1.

\begin{lem}
   \label{lem:h1-interp-stable-1d}
   For all $u \in \QQ_p$, it holds that $| \II_{p-1} u |_1 \leq | u |_1$.
\end{lem}
\begin{proof}
   As in the proof of \Cref{lem:l2-interp-stable-1d}, we write $u_{p-1} = \II_{p-1} u$, and $u = u_{p-1} + c \chi$ for some constant $c$.
   From \eqref{eq:chi-derivative}, recall that $\chi'(x) = (p^2 - p)P_{p-1}(x)$, and so
   \begin{align*}
      | u |_1^2
         &= \int_{-1}^1 \left( u_{p-1}'(t) + c \chi'(t) \right)^2 \, dt \\
         &= \int_{-1}^1 \left( u_{p-1}'(t) + c (p^2 - p) P_{p-1} \right)^2 \, dt \\
         &= | u_{p-1} |_1^2 + 2 c^2 \frac{(p^2-p)^2}{2p - 1},
   \end{align*}
   from which the result follows.
\end{proof}

We now turn to the $H^1$-stability of interpolation at Gauss--Lobatto nodes in $d$ dimensions.
This result was proven for $d=2$ and $d=3$ in \cite{Pavarino1992,Pavarino1993};
presently we provide a $d$-dimensional proof with explicit constants.

\begin{lem*}[Restatement of \Cref{lem:h1-interp-stability}]
   Consider the space $\QQ_p^d$ (resp.~$\QQ_{p-1}^d$) of $d$-variate polynomials of degree at most $p$ (resp.~$p-1$) in each variable.
   Let $\II_{p-1} : \QQ_p \to \QQ_{p-1}$ denote the interpolation operator at the Gauss--Lobatto points, and let $\II_{p-1}^\oo{d} : \QQ_p^d \to \QQ_{p-1}^d$ denote its $d$-dimensional tensor product,
   \[
      \II_{p-1}^\oo{d} = \II_{p-1} \otimes \II_{p-1} \otimes \cdots \otimes \II_{p-1}.
   \]
   Then, for all $u \in \QQ_p^d$, it holds that
   \begin{align}
      \label{eq:nd-l2-stability}
      \| \II_{p-1}^\oo{d} u \|_0^2 &\leq \left( 2 + \frac{4}{2p-3} \right)^d \| u \|_0^2, \\
      \label{eq:nd-h1-stability}
      | \II_{p-1}^\oo{d} u |_1^2 &\leq \left( 2 + \frac{4}{2p-3} \right)^{d-1} | u |_1^2.
   \end{align}
\end{lem*}
\begin{proof}
   Let $\mathsf{M}$ denote the one-dimensional mass matrix on the space $\QQ_p$ in a given basis.
   Then, $\mathsf{M}$ satisfies $\| u \|_0^2 = \mathsf{u}^\tr \mathsf{M} \mathsf{u},$
   where $\mathsf{u}$ is the vector representation of $u$ in the same basis.
   Likewise, let $\mathsf{L}$ be the one-dimensional stiffness matrix, satisfying $| u |_1^2 = \mathsf{u}^\tr \mathsf{L} \mathsf{u}$.
   Let $\mathsf{J}$ be the matrix representation of $\II_{p-1}$, and let $\widehat{\mathsf{M}}$ and $\widehat{\mathsf{L}}$ denote the mass and stiffness matrices on $\QQ_{p-1}$.
   Then, the generalized eigenvalues
   \begin{align}
      \label{eq:1d-generalized-eigvals}
      \mathsf{J}^\tr \widehat{\mathsf{M}} \mathsf{J} \mathsf{u} = \mu \mathsf{M} \mathsf{u}, \qquad\text{and}\qquad
      \mathsf{J}^\tr \widehat{\mathsf{L}} \mathsf{J} \mathsf{u} = \lambda \mathsf{L} \mathsf{u},
   \end{align}
   satisfy
   \[
      \mu \leq 2 + \frac{4}{2p-3}, \qquad\text{and}\qquad
      \lambda \leq 1.
   \]
   The bound on $\mu$ follows from \Cref{lem:l2-interp-stable-1d} and the bound on $\lambda$ follows from \Cref{lem:h1-interp-stable-1d}.

   Note that for $u \in \QQ_{p-1}^d$, it then holds that
   \begin{align*}
      \| u \|_0^2 &= (\mathsf{M} \otimes \mathsf{M} \otimes \cdots \otimes \mathsf{M}) \mathsf{u} = \mathsf{M}^\oo{d} \mathsf{u}.
   \end{align*}
   Also, the matrix representation of $\II_{p-1}^\oo{d}$ is given by $\mathsf{J}^\oo{d} = \mathsf{J} \otimes \mathsf{J} \otimes \cdots \otimes \mathsf{J}$.
   Then,
   \[
      (\mathsf{J}^\oo{d})^\tr \widehat{\mathsf{M}}^\oo{d} \mathsf{J}^\oo{d}
      = (\mathsf{J}^\tr \widehat{\mathsf{M}} \mathsf{J}) \otimes (\mathsf{J}^\tr \widehat{\mathsf{M}} \mathsf{J})
      \otimes \cdots \otimes (\mathsf{J}^\tr \widehat{\mathsf{M}} \mathsf{J}),
   \]
   which implies that the generalized eigenvalues $\nu$
   \[
      (\mathsf{J}^\oo{d})^\tr \widehat{\mathsf{M}}^\oo{d} \mathsf{J}^\oo{d} \mathsf{u} = \nu \mathsf{M}^\oo{d} \mathsf{u}
   \]
   are given by the $d$-fold products of the one-dimensional generalized eigenvalues $\mu$, from which \eqref{eq:nd-l2-stability} follows.

   The $H^1$ seminorm in $d$ dimensions is induced by the $d$-dimensional stiffness matrix
   \[
      \mathsf{L}_d =
      (\mathsf{L} \otimes \mathsf{M} \otimes \cdots \otimes \mathsf{M}) +
      (\mathsf{M} \otimes \mathsf{L} \otimes \cdots \otimes \mathsf{M}) +
      \cdots +
      (\mathsf{M} \otimes \mathsf{M} \otimes \cdots \otimes \mathsf{L}),
   \]
   and similarly for $\mathsf{\widehat{L}}_d$.
   As above,
   \begin{multline*}
      (\mathsf{J}^\oo{d})^\tr \widehat{\mathsf{L}}_d \mathsf{J}^\oo{d} =
      (\mathsf{J}^\tr \widehat{\mathsf{L}} \mathsf{J} \otimes \mathsf{J}^\tr\widehat{\mathsf{M}} \mathsf{J} \otimes \cdots \otimes \mathsf{J}^\tr\widehat{\mathsf{M}}\mathsf{J}) \\ +
      (\mathsf{J}^\tr\widehat{\mathsf{M}} \mathsf{J} \otimes \mathsf{J}^\tr\widehat{\mathsf{L}} \mathsf{J} \otimes \cdots \otimes \mathsf{J}^\tr\widehat{\mathsf{M}}\mathsf{J}) +
      \cdots +
      (\mathsf{J}^\tr\widehat{\mathsf{M}} \mathsf{J} \otimes \mathsf{J}^\tr\widehat{\mathsf{M}} \mathsf{J} \otimes \cdots \otimes \mathsf{J}^\tr\widehat{\mathsf{L}}\mathsf{J}).
   \end{multline*}
   Each summand contains $d-1$ mass matrix terms, and one stiffness matrix term, and so the generalized eigenvalues for pairs of matrices of the form $\mathsf{J}^\tr \widehat{\mathsf{L}} \mathsf{J} \otimes \mathsf{J}^\tr\widehat{\mathsf{M}} \mathsf{J} \otimes \cdots \otimes \mathsf{J}^\tr\widehat{\mathsf{M}}\mathsf{J}$ and $\mathsf{L} \otimes \mathsf{M} \otimes \cdots \otimes \mathsf{M}$ are given by the product of $d-1$ of the generalized eigenvalues $\mu$ from \eqref{eq:1d-generalized-eigvals}.
   Thus,
   \[
      \mathsf{u}^\tr (\mathsf{J}^\tr \mathsf{L} \mathsf{J} \otimes \mathsf{J}^\tr\mathsf{M} \mathsf{J} \otimes \cdots \otimes \mathsf{J}^\tr\mathsf{M}\mathsf{J}) \mathsf{u} \leq \left(2 + \frac{4}{2p-3}\right)^{d-1} \mathsf{u}^\tr \left( \mathsf{L} \otimes \mathsf{M} \otimes \cdots \otimes \mathsf{M} \right) \mathsf{u}.
   \]
   Since this bound holds for every such summand, we then have that
   \[
      \mathsf{u}^\tr (\mathsf{J}^\oo{d})^\tr \mathsf{L}_d \mathsf{J}^\oo{d} \mathsf{u} \leq \left(2 + \frac{4}{2p-3}\right)^{d-1} \mathsf{u}^\tr \mathsf{L}_d \mathsf{u},
   \]
   proving \eqref{eq:nd-h1-stability}.
\end{proof}

\begin{lem*}[Restatement of \Cref{lem:p-pm1-interp}]
   For any $u \in \QQ_p$ it holds that
   \[
      \| u - \II_{p-1} u \|_0^2 \leq \frac{2}{4p^2 - 4p - 3} | u |_{1}^2.
   \]
\end{lem*}
\begin{proof}
   Expanding $u = \sum_{i=0}^p a_i P_i$ and given the expansion \eqref{eq:chi} of $\chi = \chi_p$ results in $u - \II_{p-1} u = a_p\left( \frac{2p-1}{p^2-p} \right) \chi$.
   Therefore, using the expression \eqref{eq:chi-norm} for $\| \chi \|_0^2$,
   \[
      \| u - \II_{p-1} u \|_0^2
         = a_p^2 \left( \frac{2p-1}{p^2-p} \right)^2 \| \chi \|_0^2
         = a_p^2 \left( \frac{2}{2p+1} + \frac{2}{2p-3} \right).
   \]
   We recall the identity
   \[
      P_i' = \sum_{j=1}^{\lceil i/2 \rceil} \frac{2 P_{i + 1 - 2j}}{ \| P_{i + 1 - 2j} \|_0^2 },
   \]
   and so
   \[
      | u |_1^2 \geq a_p^2 \frac{4}{\| P_{p-1} \|_0^2} = a_p^2 (4p - 2)
   \]
   from which the result follows by simple algebra.
\end{proof}

\section{Numerical Results}
\label{sec:results}

In this section, we numerically study the performance of the preconditioners described in \Cref{sec:subspace,,sec:fictitious,,sec:auxiliary}. 
The subspace correction preconditioner with vertex patches is denoted $\Bsub$, the fictitious space preconditioner is denoted $\Bfic$, and the auxiliary space preconditioner is denoted $\Baux$.
The discretization and solvers have been implemented in the open-source MFEM finite element software \cite{Anderson2020,Andrej2024}.
Where indicated, an algebraic multigrid V-cycle is used to approximate the inverse of global problems (i.e.~the $H^1$-conforming subspace, fictitious space, or auxiliary space problems); we use BoomerAMG from the \textit{hypre} library for this purpose \cite{Henson2002}.

\subsection{Cartesian grid}

As a first test case, we consider an $n \times n$ Cartesian grid, and study solver performance with respect to $n$, polynomial degree $p$, and penalty parameter $\eta$.
In this case, all elements are rectangular, and so \Cref{thm:subspace-conditioning} and \Cref{thm:fictitious} guarantee uniform convergence of the subspace and fictitious space preconditioners, independent of discretization parameters.

In \Cref{tab:cartesian-cond}, we present the computed condition numbers of system matrix $A$ and the preconditioned systems $\Bsub A$ and $\Bfic A$, where $\Bsub$ is the subspace correction preconditioner of \Cref{sec:subspace}, and $\Bfic$ is the auxiliary space preconditioner of \Cref{sec:fictitious}.
Here, we compute the exact inverses of the subspace and auxiliary space problems using direct methods;
for larger problems this is not feasible.
We consider both fixing the polynomial degree and increasing the mesh refinement, and fixing the mesh size and increasing the polynomial degree.
In both cases, the condition numbers of the preconditioner systems remain bounded, whereas the condition number of the original system grows considerably.
Additionally, we consider two cases for the penalty parameter: $\eta = 10$ and $\eta = 10^4$.
Linear scaling of the unpreconditioned system with respect to $\eta$ is observed, in accordance with \Cref{lem:eigval-estimates}.
In contrast, the condition numbers of the preconditioned systems are roughly unchanged.

\begin{table}
   \caption{
      Computed condition numbers for the Cartesian grid test case using subspace preconditioner $\Bsub$ and auxiliary space preconditioner $\Bfic$.
      Dependence on mesh refinement, polynomial degree $p$, and penalty parameter $\eta$ are shown.
   }
   \label{tab:cartesian-cond}

   \begin{tabular}{cccccc}
      \multicolumn{6}{c}{$\eta = 10$} \\[3pt]
      \toprule
      & $n$ & $\kappa(A)$ & $\kappa(\Bsub A)$ & $\kappa(\Bfic A)$ & $\kappa(\Baux F)$ \\
      \midrule
      $p = 2$ & 4  & $1.67\times10^{3}$ & $5.92$ & $2.92$ & $3.80$ \\
      $p = 2$ & 8  & $6.74\times10^{3}$ & $5.92$ & $3.33$ & $4.73$ \\
      $p = 2$ & 16 & $2.70\times10^{4}$ & $5.92$ & $3.46$ & $5.02$ \\
      \midrule
      $p = 3$ & 4  & $2.13\times10^{3}$ & $5.96$ & $1.88$ & $4.24$ \\
      $p = 4$ & 4  & $2.66\times10^{3}$ & $5.98$ & $1.68$ & $4.39$ \\
      $p = 5$ & 4  & $4.07\times10^{3}$ & $5.98$ & $1.60$ & $4.32$ \\
      \bottomrule
   \end{tabular}

   \vspace{0.5cm}

   \begin{tabular}{cccccc}
      \multicolumn{6}{c}{$\eta = 10^4$} \\[3pt]
      \toprule
      & $n$ & $\kappa(A)$ & $\kappa(\Bsub A)$ & $\kappa(\Bfic A)$ & $\kappa(\Baux F)$ \\
      \midrule
      $p = 2$ & 4 & $1.66\times10^{6}$ & $6.00$ & $2.84$ & $3.93$ \\
      $p = 2$ & 8 & $6.67\times10^{6}$ & $6.72$ & $3.31$ & $4.85$ \\
      $p = 2$ & 16 & $2.67\times10^{7}$ & $7.07$ & $3.45$ & $5.14$ \\
      \midrule
      $p = 3$ & 4 & $2.13\times10^{6}$ & $6.00$ & $1.84$ & $4.31$ \\
      $p = 4$ & 4 & $2.69\times10^{6}$ & $6.00$ & $1.65$ & $4.43$ \\
      $p = 5$ & 4 & $4.13\times10^{6}$ & $6.00$ & $1.58$ & $4.32$ \\
      \bottomrule
   \end{tabular}
\end{table}

\begin{table}
   \caption{Conjugate gradient iteration counts (with relative tolerance of $10^{-12}$) for the Cartesian grid test case with subspace, fictitious space, and auxiliary space preconditioners.}
   \label{tab:cartesian-iters}

   \aboverulesep = 0pt
   \belowrulesep = 0pt

   \resizebox{\linewidth}{!}{%
   \begin{tabular}{|cY?cc|cc|cc?cc|cc|cc|}
      \toprule
      \tabstrut & \multicolumn{1}{c}{} & \multicolumn{6}{c?}{$p=2\qquad\eta=1$}&\multicolumn{6}{c|}{$p=2\qquad\eta=100$} \\
      \midrule[1pt]
      \multicolumn{2}{|c?}{} & \multicolumn{2}{c|}{$\Bsub A$} & \multicolumn{2}{c}{$\Bfic A$} & \multicolumn{2}{c?}{$\Baux A$} & \multicolumn{2}{c|}{$\Bsub A$} & \multicolumn{2}{c}{$\Bfic A$} & \multicolumn{2}{c|}{$\Baux A$} \\
      $n$ & {\# DOFs} & It. & Time (s) & It. & Time (s) & It. & Time (s) & It. & Time (s) & It. & Time (s) & It. & Time (s) \\
      \midrule[1pt]
      4  &   144 & 17 & 0.0002 & 25 & 0.001 & 25 & 0.0004 & 19 & 0.0002 & 25 & 0.000 & 24 & 0.0004 \\
      8  &   544 & 24 & 0.001 & 29 & 0.002 & 28 & 0.001 & 26 & 0.001 & 30 & 0.002 & 28 & 0.001 \\
      16 &  2112 & 32 & 0.004 & 30 & 0.008 & 29 & 0.005 & 32 & 0.004 & 32 & 0.008 & 28 & 0.005 \\
      32 &  8320 & 39 & 0.020 & 30 & 0.032 & 29 & 0.018 & 38 & 0.020 & 33 & 0.034 & 28 & 0.018 \\
      \midrule[1pt]
      \tabstrut & \multicolumn{1}{c}{} & \multicolumn{6}{c?}{$p=3\qquad\eta=1$}&\multicolumn{6}{c|}{$p=3\qquad\eta=100$} \\
      \midrule[1pt]
      4  &   312 & 20 & 0.001 & 23 & 0.001 & 26 & 0.001 & 21 & 0.001 & 21 & 0.001 & 26 & 0.001 \\
      8  &  1200 & 26 & 0.003 & 24 & 0.004 & 28 & 0.003 & 26 & 0.003 & 23 & 0.004 & 28 & 0.003 \\
      16 &  4704 & 35 & 0.016 & 24 & 0.017 & 29 & 0.013 & 32 & 0.014 & 23 & 0.016 & 30 & 0.014 \\
      32 & 18624 & 42 & 0.078 & 24 & 0.069 & 30 & 0.057 & 38 & 0.069 & 23 & 0.084 & 30 & 0.055 \\
      \midrule[1pt]
      \tabstrut & \multicolumn{1}{c}{} & \multicolumn{6}{c?}{$p=4\qquad\eta=1$}&\multicolumn{6}{c|}{$p=4\qquad\eta=100$} \\
      \midrule[1pt]
      4  &   544 & 21 & 0.001 & 26 & 0.002 & 29 & 0.002 & 21 & 0.001 & 23 & 0.002 & 26 & 0.002 \\
      8  &  2112 & 26 & 0.008 & 27 & 0.007 & 29 & 0.008 & 25 & 0.008 & 24 & 0.006 & 27 & 0.007 \\
      16 &  8320 & 35 & 0.046 & 27 & 0.031 & 29 & 0.032 & 32 & 0.041 & 25 & 0.028 & 27 & 0.029 \\
      32 & 33024 & 42 & 0.230 & 27 & 0.127 & 29 & 0.131 & 38 & 0.208 & 25 & 0.120 & 27 & 0.119 \\
      \midrule[1pt]
      \tabstrut & \multicolumn{1}{c}{} & \multicolumn{6}{c?}{$p=5\qquad\eta=1$}&\multicolumn{6}{c|}{$p=5\qquad\eta=100$} \\
      \midrule[1pt]
      4  &   840 & 21 & 0.003 & 29 & 0.004 & 31 & 0.003 & 21 & 0.003 & 22 & 0.003 & 30 & 0.003 \\
      8  &  3280 & 27 & 0.018 & 29 & 0.016 & 34 & 0.015 & 25 & 0.017 & 22 & 0.011 & 33 & 0.015 \\
      16 & 12960 & 35 & 0.139 & 29 & 0.065 & 35 & 0.070 & 31 & 0.092 & 22 & 0.049 & 34 & 0.063 \\
      32 & 51520 & 43 & 0.570 & 29 & 0.257 & 36 & 0.277 & 38 & 0.524 & 22 & 0.199 & 34 & 0.260 \\
      \midrule[1pt]
      \tabstrut & \multicolumn{1}{c}{} & \multicolumn{6}{c?}{$p=6\qquad\eta=1$}&\multicolumn{6}{c|}{$p=6\qquad\eta=100$} \\
      \midrule[1pt]
      4  &  1200 & 22 & 0.007 & 33 & 0.007 & 33 & 0.007 & 21 & 0.007 & 24 & 0.005 & 27 & 0.005 \\
      8  &  4704 & 27 & 0.040 & 33 & 0.029 & 34 & 0.028 & 25 & 0.036 & 25 & 0.022 & 28 & 0.022 \\
      16 & 18624 & 35 & 0.205 & 33 & 0.163 & 34 & 0.118 & 31 & 0.185 & 26 & 0.094 & 28 & 0.095 \\
      32 & 74112 & 43 & 1.047 & 32 & 0.480 & 34 & 0.461 & 38 & 1.009 & 26 & 0.384 & 28 & 0.385 \\
      \bottomrule
   \end{tabular}}
\end{table}

In \Cref{tab:cartesian-iters}, we consider mesh refinement for polynomial degree $p \in \{ 2, 3, 4, 5 \}$.
Since these problems are larger, direct solvers are not practical for the global problems (the coarse subspace and the auxiliary space).
The inverse of the coarse problem $A_0$ in the subspace preconditioner is replaced with one V-cycle of \textit{hypre}'s BoomerAMG algebraic multigrid.
The inverse of the discontinuous Galerkin fictitious space and auxiliary space problems $\widetilde{A}$ and $A_0$ are replaced with one V-cycle of algebraic multigrid constructed with the spectrally equivalent low-order-refined discretization, as described in \Cref{sec:matrix-free}.
In \Cref{tab:cartesian-iters}, we report solver runtime and the number of conjugate iterations required to achieve a relative tolerance of $10^{-12}$.
The subspace and auxiliary space preconditioners exhibit overall similar runtimes; however, for higher-order problems, the fictitious space solver is the fastest, and the subspace correction preconditioner is the slowest.
This can be attributed to the local vertex patch solvers required by $\Bsub$, which scale poorly with $p$.
The preconditioners $\Bsub$ and $\Baux$ exhibit a slight preasymptotic increase in the number of iterations with increasing mesh refinement; this can be attributed to clustering of eigenvalues in problems with smaller sizes, and is consistent with the condition numbers computed in \Cref{tab:cartesian-cond}.

\subsection{Unstructured mesh with affine elements}
\label{sec:unstructured-affine}

In this test case, we consider an unstructured mesh consisting of only affine elements (parallelograms).
For this mesh, the elements are arranged to form a star, as depicted in the left panel of \Cref{fig:meshes}..
We consider scaling of the preconditioners with respect to polynomial degree, and under uniform refinement of the mesh.
The conjugate gradient iteration counts and runtimes are shown in \Cref{tab:affine-iters}.
A relative tolerance of $10^{-12}$ is used for the CG stopping criterion.
Two cases are considered for the penalty parameter: $\eta = 1$ and $\eta = 100$.
On this mesh, we expect $\kappa(\Bsub A)$ and $\kappa(\Baux A)$ to be independent of $\eta$, whereas $\kappa(\Bfic A)$ is bounded only up to a linear factor in $\eta$.
All three solvers exhibit iteration counts that appear relatively robust with respect to polynomial degree and mesh refinement.
As in the previous test case, slight preasymptotic scaling is observed under mesh refinement.
The behavior of all three preconditioners is qualitatively similar.
For higher-order problems, the fictitious space and auxiliary space preconditioners display faster runtimes than the subspace correction preconditioner.

\begin{figure}
   \begin{tikzpicture}
   \draw (0,0) -- (1,0);
   \draw (1,0) -- (1.30902,0.951057);
   \draw (0.309017,0.951057) -- (1.30902,0.951057);
   \draw (0,0) -- (0.309017,0.951057);
   \draw (0.309017,0.951057) -- (-0.5,1.53884);
   \draw (-0.809017,0.587785) -- (-0.5,1.53884);
   \draw (0,0) -- (-0.809017,0.587785);
   \draw (-0.809017,0.587785) -- (-1.61803,0);
   \draw (-0.809017,-0.587785) -- (-1.61803,0);
   \draw (0,0) -- (-0.809017,-0.587785);
   \draw (-0.809017,-0.587785) -- (-0.5,-1.53884);
   \draw (0.309017,-0.951057) -- (-0.5,-1.53884);
   \draw (0,0) -- (0.309017,-0.951057);
   \draw (0.309017,-0.951057) -- (1.30902,-0.951057);
   \draw (1,0) -- (1.30902,-0.951057);

   \node at (0,-2.25) {(a) $\ell = 0$};

   \begin{scope}[shift={(4.0,0)}]
      \draw (0,0) -- (0.5,0);
      \draw (0.5,0) -- (0.654509,0.475529);
      \draw (0.154508,0.475529) -- (0.654509,0.475529);
      \draw (0,0) -- (0.154508,0.475529);
      \draw (1,0) -- (0.5,0);
      \draw (1,0) -- (1.15451,0.475529);
      \draw (1.15451,0.475529) -- (0.654509,0.475529);
      \draw (1.30902,0.951057) -- (1.15451,0.475529);
      \draw (1.30902,0.951057) -- (0.809019,0.951057);
      \draw (0.809019,0.951057) -- (0.654509,0.475529);
      \draw (0.309017,0.951057) -- (0.809019,0.951057);
      \draw (0.309017,0.951057) -- (0.154508,0.475529);
      \draw (0.154508,0.475529) -- (-0.25,0.769421);
      \draw (-0.404508,0.293893) -- (-0.25,0.769421);
      \draw (0,0) -- (-0.404508,0.293893);
      \draw (0.309017,0.951057) -- (-0.0954915,1.24495);
      \draw (-0.0954915,1.24495) -- (-0.25,0.769421);
      \draw (-0.5,1.53884) -- (-0.0954915,1.24495);
      \draw (-0.5,1.53884) -- (-0.654508,1.06331);
      \draw (-0.654508,1.06331) -- (-0.25,0.769421);
      \draw (-0.809017,0.587785) -- (-0.654508,1.06331);
      \draw (-0.809017,0.587785) -- (-0.404508,0.293893);
      \draw (-0.404508,0.293893) -- (-0.809016,0);
      \draw (-0.404508,-0.293893) -- (-0.809016,0);
      \draw (0,0) -- (-0.404508,-0.293893);
      \draw (-0.809017,0.587785) -- (-1.21352,0.293893);
      \draw (-1.21352,0.293893) -- (-0.809016,0);
      \draw (-1.61803,0) -- (-1.21352,0.293893);
      \draw (-1.61803,0) -- (-1.21352,-0.293893);
      \draw (-1.21352,-0.293893) -- (-0.809016,0);
      \draw (-0.809017,-0.587785) -- (-1.21352,-0.293893);
      \draw (-0.809017,-0.587785) -- (-0.404508,-0.293893);
      \draw (-0.404508,-0.293893) -- (-0.25,-0.76942);
      \draw (0.154508,-0.475529) -- (-0.25,-0.76942);
      \draw (0,0) -- (0.154508,-0.475529);
      \draw (-0.809017,-0.587785) -- (-0.654508,-1.06331);
      \draw (-0.654508,-1.06331) -- (-0.25,-0.76942);
      \draw (-0.5,-1.53884) -- (-0.654508,-1.06331);
      \draw (-0.5,-1.53884) -- (-0.0954915,-1.24495);
      \draw (-0.0954915,-1.24495) -- (-0.25,-0.76942);
      \draw (0.309017,-0.951057) -- (-0.0954915,-1.24495);
      \draw (0.309017,-0.951057) -- (0.154508,-0.475529);
      \draw (0.154508,-0.475529) -- (0.654509,-0.475529);
      \draw (0.5,0) -- (0.654509,-0.475529);
      \draw (0.309017,-0.951057) -- (0.809019,-0.951057);
      \draw (0.809019,-0.951057) -- (0.654509,-0.475529);
      \draw (1.30902,-0.951057) -- (0.809019,-0.951057);
      \draw (1.30902,-0.951057) -- (1.15451,-0.475529);
      \draw (1.15451,-0.475529) -- (0.654509,-0.475529);
      \draw (1,0) -- (1.15451,-0.475529);

      \node at (0,-2.25) {(b) $\ell = 1$};
   \end{scope}

\end{tikzpicture}
   \hspace{1cm}
   \input{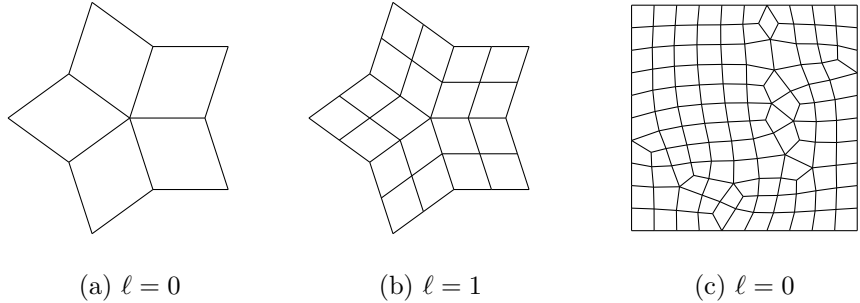}
   \caption{Panels (a) and (b): unstructured mesh of affine elements (parallelograms) and uniform refinement levels $\ell$ used for the test case in \Cref{sec:unstructured-affine}.
   Panel (c): unstructured with with non-affine (skewed) elements used for the test case in \Cref{sec:unstructured-non-affine}.}
   \label{fig:meshes}.
\end{figure}

\begin{table}
   \caption{Conjugate gradient iteration counts (with relative tolerance of $10^{-12}$) for the affine mesh test case with subspace, fictitious space, and auxiliary space preconditioners.}
   \label{tab:affine-iters}

   \aboverulesep = 0pt
   \belowrulesep = 0pt

   \resizebox{\linewidth}{!}{%
   \begin{tabular}{|cY?cc|cc|cc?cc|cc|cc|}
      \toprule
      \tabstrut & \multicolumn{1}{c}{} & \multicolumn{6}{c?}{$p=2\qquad\eta=1$}&\multicolumn{6}{c|}{$p=2\qquad\eta=100$} \\
      \midrule[1pt]
      \multicolumn{2}{|c?}{} & \multicolumn{2}{c|}{$\Bsub A$} & \multicolumn{2}{c}{$\Bfic A$} & \multicolumn{2}{c?}{$\Baux A$} & \multicolumn{2}{c|}{$\Bsub A$} & \multicolumn{2}{c}{$\Bfic A$} & \multicolumn{2}{c|}{$\Baux A$} \\
      $\ell$ & {\# DOFs} & It. & Time (s) & It. & Time (s) & It. & Time (s) & It. & Time (s) & It. & Time (s) & It. & Time (s) \\
      \midrule[1pt]
      0  &    50 & 16 & 0.0002 & 20 & 0.0002 & 22 & 0.0001 & 13 & 0.0001 & 27 & 0.0003 & 25 & 0.0001 \\
      1  &   180 & 22 & 0.0005 & 28 & 0.001 & 30 & 0.001 & 24 & 0.0003 & 53 & 0.001 & 37 & 0.001 \\
      2  &   680 & 29 & 0.001 & 32 & 0.003 & 34 & 0.002 & 28 & 0.001 & 78 & 0.007 & 42 & 0.002 \\
      3  &  2640 & 35 & 0.006 & 34 & 0.012 & 40 & 0.008 & 34 & 0.007 & 99 & 0.034 & 47 & 0.015 \\
      4  & 10400 & 43 & 0.029 & 36 & 0.060 & 44 & 0.034 & 41 & 0.028 & 111 & 0.144 & 46 & 0.038 \\
      \midrule[1pt]
      \tabstrut & \multicolumn{1}{c}{} & \multicolumn{6}{c?}{$p=3\qquad\eta=1$}&\multicolumn{6}{c|}{$p=3\qquad\eta=100$} \\
      \midrule[1pt]
      0  &   105 & 20 & 0.0002 & 25 & 0.0004 & 27 & 0.0003 & 16 & 0.0002 & 27 & 0.0004 & 27 & 0.0003 \\
      1  &   390 & 24 & 0.001 & 26 & 0.002 & 30 & 0.001 & 24 & 0.001 & 43 & 0.002 & 31 & 0.001 \\
      2  &  1500 & 30 & 0.004 & 27 & 0.006 & 32 & 0.005 & 29 & 0.004 & 55 & 0.033 & 35 & 0.005 \\
      3  &  5880 & 37 & 0.021 & 27 & 0.023 & 34 & 0.021 & 35 & 0.022 & 65 & 0.055 & 37 & 0.021 \\
      4  & 23280 & 44 & 0.101 & 27 & 0.097 & 35 & 0.082 & 41 & 0.101 & 68 & 0.234 & 39 & 0.093 \\
      \midrule[1pt]
      \tabstrut & \multicolumn{1}{c}{} & \multicolumn{6}{c?}{$p=4\qquad\eta=1$}&\multicolumn{6}{c|}{$p=4\qquad\eta=100$} \\
      \midrule[1pt]
      0  &   180 & 21 & 0.001 & 28 & 0.001 & 30 & 0.001 & 17 & 0.001 & 37 & 0.001 & 26 & 0.001 \\
      1  &   680 & 25 & 0.002 & 28 & 0.003 & 31 & 0.003 & 25 & 0.002 & 50 & 0.004 & 29 & 0.002 \\
      2  &  2640 & 30 & 0.033 & 29 & 0.010 & 31 & 0.010 & 28 & 0.011 & 64 & 0.022 & 31 & 0.010 \\
      3  & 10400 & 38 & 0.061 & 30 & 0.044 & 32 & 0.043 & 34 & 0.088 & 72 & 0.098 & 31 & 0.042 \\
      4  & 41280 & 45 & 0.306 & 30 & 0.181 & 32 & 0.177 & 41 & 0.274 & 81 & 0.458 & 32 & 0.179 \\
      \midrule[1pt]
      \tabstrut & \multicolumn{1}{c}{} & \multicolumn{6}{c?}{$p=5\qquad\eta=1$}&\multicolumn{6}{c|}{$p=5\qquad\eta=100$} \\
      \midrule[1pt]
      0  &   275 & 21 & 0.001 & 31 & 0.001 & 32 & 0.001 & 17 & 0.001 & 39 & 0.002 & 28 & 0.001 \\
      1  &  1050 & 26 & 0.006 & 31 & 0.005 & 35 & 0.005 & 25 & 0.005 & 53 & 0.009 & 33 & 0.006 \\
      2  &  4100 & 31 & 0.029 & 31 & 0.022 & 38 & 0.024 & 28 & 0.025 & 62 & 0.043 & 36 & 0.027 \\
      3  & 16200 & 38 & 0.140 & 31 & 0.085 & 40 & 0.093 & 34 & 0.133 & 70 & 0.187 & 38 & 0.087 \\
      4  & 64400 & 45 & 0.696 & 30 & 0.332 & 41 & 0.380 & 41 & 0.614 & 71 & 0.852 & 39 & 0.381 \\
      \midrule[1pt]
      \tabstrut & \multicolumn{1}{c}{} & \multicolumn{6}{c?}{$p=6\qquad\eta=1$}&\multicolumn{6}{c|}{$p=6\qquad\eta=100$} \\
      \midrule[1pt]
      0  &   390 & 22 & 0.002 & 34 & 0.002 & 33 & 0.002 & 18 & 0.002 & 41 & 0.003 & 31 & 0.002 \\
      1  &  1500 & 26 & 0.011 & 34 & 0.009 & 36 & 0.009 & 25 & 0.010 & 54 & 0.014 & 35 & 0.008 \\
      2  &  5880 & 31 & 0.055 & 34 & 0.037 & 38 & 0.040 & 28 & 0.048 & 67 & 0.081 & 36 & 0.038 \\
      3  & 23280 & 38 & 0.273 & 34 & 0.153 & 38 & 0.159 & 34 & 0.248 & 73 & 0.370 & 36 & 0.157 \\
      4  & 92640 & 45 & 1.339 & 35 & 0.661 & 38 & 0.677 & 41 & 1.293 & 79 & 1.418 & 36 & 0.624 \\
      \bottomrule
   \end{tabular}}
\end{table}

\subsection{Unstructured mesh with non-affine elements}
\label{sec:unstructured-non-affine}

In this test case, we consider an unstructured mesh with non-affined (skewed) elements, shown in the right-panel of \Cref{fig:meshes}.
The conjugate gradient iteration counts and solver runtimes are shown in \Cref{tab:unstructured-iters}.
In contrast to the two previous test cases, applied to this mesh, the convergence estimates of \Cref{thm:subspace-conditioning,,thm:fictitious,,thm:auxiliary} do not guarantee $p$- and $\eta$-independent conditioning.
Despite the lack of theoretical guarantees, the iteration scaling with respect to polynomial degree appears very mild, particularly for the subspace and auxiliary space preconditioners $\Bsub$ and $\Baux$.
The increase in the condition number with increasing penalty parameter is most noticeable for the fictitious space preconditioner $\Bfic$, for which a severe increase in iterations is observed.
For this problem, and especially at higher orders, $\Baux$ provides the fastest time to solution.

\begin{table}
   \caption{Conjugate gradient iteration counts (with relative tolerance of $10^{-12}$) for the unstructured (non-affine) test case with subspace, fictitious space, and auxiliary space preconditioners.}
   \label{tab:unstructured-iters}

   \aboverulesep = 0pt
   \belowrulesep = 0pt

   \resizebox{\linewidth}{!}{%
   \begin{tabular}{|cZ?cc|cc|cc?cc|cc|cc|}
      \toprule
      \tabstrut & \multicolumn{1}{c}{} & \multicolumn{6}{c?}{$p=2\qquad\eta=1$}&\multicolumn{6}{c|}{$p=2\qquad\eta=100$} \\
      \midrule[1pt]
      \multicolumn{2}{|c?}{} & \multicolumn{2}{c|}{$\Bsub A$} & \multicolumn{2}{c}{$\Bfic A$} & \multicolumn{2}{c?}{$\Baux A$} & \multicolumn{2}{c|}{$\Bsub A$} & \multicolumn{2}{c}{$\Bfic A$} & \multicolumn{2}{c|}{$\Baux A$} \\
      $\ell$ & {\# DOFs} & It. & Time (s) & It. & Time (s) & It. & Time (s) & It. & Time (s) & It. & Time (s) & It. & Time (s) \\
      \midrule[1pt]
      0  &   992 & 32 & 0.002 & 44 & 0.005 & 37 & 0.003 & 32 & 0.002 & 260 & 0.031 & 52 & 0.004 \\
      1  &  3888 & 43 & 0.011 & 47 & 0.025 & 42 & 0.012 & 40 & 0.010 & 342 & 0.174 & 50 & 0.016 \\
      2  & 15392 & 48 & 0.050 & 54 & 0.121 & 45 & 0.057 & 45 & 0.058 & 441 & 0.892 & 50 & 0.061 \\
      \midrule[1pt]
      \tabstrut & \multicolumn{1}{c}{} & \multicolumn{6}{c?}{$p=3\qquad\eta=1$}&\multicolumn{6}{c|}{$p=3\qquad\eta=100$} \\
      \midrule[1pt]
      0  &  2202 & 33 & 0.007 & 40 & 0.012 & 35 & 0.008 & 32 & 0.007 & 238 & 0.072 & 46 & 0.010 \\
      1  &  8688 & 44 & 0.040 & 43 & 0.060 & 35 & 0.031 & 42 & 0.034 & 302 & 0.395 & 40 & 0.039 \\
      2  & 34512 & 49 & 0.173 & 47 & 0.255 & 39 & 0.143 & 46 & 0.160 & 379 & 2.009 & 44 & 0.155 \\
      \midrule[1pt]
      \tabstrut & \multicolumn{1}{c}{} & \multicolumn{6}{c?}{$p=4\qquad\eta=1$}&\multicolumn{6}{c|}{$p=4\qquad\eta=100$} \\
      \midrule[1pt]
      0  &  3888 & 34 & 0.019 & 42 & 0.022 & 35 & 0.016 & 31 & 0.020 & 250 & 0.123 & 42 & 0.081 \\
      1  & 15392 & 45 & 0.114 & 49 & 0.108 & 35 & 0.071 & 40 & 0.098 & 352 & 0.739 & 37 & 0.076 \\
      2  & 61248 & 49 & 0.563 & 57 & 0.500 & 38 & 0.319 & 45 & 0.488 & 460 & 3.941 & 38 & 0.318 \\
      \midrule[1pt]
      \tabstrut & \multicolumn{1}{c}{} & \multicolumn{6}{c?}{$p=5\qquad\eta=1$}&\multicolumn{6}{c|}{$p=5\qquad\eta=100$} \\
      \midrule[1pt]
      0  &  6050 & 33 & 0.051 & 45 & 0.045 & 39 & 0.033 & 31 & 0.048 & 256 & 0.388 & 43 & 0.054 \\
      1  & 24000 & 45 & 0.254 & 52 & 0.223 & 43 & 0.159 & 41 & 0.221 & 331 & 1.342 & 44 & 0.172 \\
      2  & 95600 & 49 & 1.124 & 57 & 0.939 & 47 & 0.903 & 46 & 1.027 & 424 & 6.918 & 49 & 0.709 \\
      \midrule[1pt]
      \tabstrut & \multicolumn{1}{c}{} & \multicolumn{6}{c?}{$p=6\qquad\eta=1$}&\multicolumn{6}{c|}{$p=6\qquad\eta=100$} \\
      \midrule[1pt]
      0  &  8688 & 33 & 0.091 & 49 & 0.079 & 41 & 0.065 & 31 & 0.095 & 281 & 0.487 & 43 & 0.068 \\
      1  & 34512 & 45 & 0.484 & 58 & 0.396 & 43 & 0.283 & 39 & 0.422 & 375 & 2.452 & 41 & 0.258 \\
      2  & 137568 & 49 & 2.455 & 64 & 1.749 & 44 & 1.179 & 46 & 2.251 & 479 & 13.126 & 45 & 1.149 \\
      \bottomrule
   \end{tabular}}
\end{table}

\subsection{Stokes problem}

As a final test case, we consider the Stokes problem \eqref{eq:stokes} on $\Omega = [0,1]^2$.
The mesh is taken to be an $n \times n$ Cartesian grid.
Dirichlet conditions for the velocity are applied at all domain boundaries; the pressure is determined up to a constant.
The computed velocity field is exactly divergence-free.
We use the block-diagonal preconditioner $\widetilde{\mathcal{B}}$ defined by approximating the blocks of \eqref{eq:block-diag}.
The approximate Schur complement inverse $\widetilde{S}^{-1} \approx S^{-1}$ is taken to be the inverse of the diagonal of the mass matrix;
this approximation is robust with respect to $h$ and $p$.
The approximation of the inverse of the $(1,1)$-block is taken to be one of the three preconditioners considered, $\Bsub$, $\Bfic$, or $\Baux$.
The saddle-point system \eqref{eq:saddle-point} is solved using MINRES with a relative tolerance of $10^{-12}$.
As in the previous cases, we consider two choices for the penalty parameter: $\eta = 1$ and $\eta = 100$.
Note that the spectral equivalence $S \approx M$ is dependent on the inf--sup constant, which itself depends on the penalty parameter.
Therefore, even for preconditioners $B \approx A^{-1}$ that are independent of $\eta$, the equivalence $\mathcal{B} \approx \mathcal{A}^{-1}$ will be $\eta$-dependent.

MINRES iteration counts are shown in \Cref{tab:stokes}.
Because of the considerations discussed above, the iteration counts are significantly higher in the $\eta = 100$ case for all three of the solvers considered.
This emphasizes the importance for this problem of choosing the penalty no larger than is required for stability of the method, for example by using the explicit expressions derived in \cite{Shahbazi2005}.
For this problem, all of the solvers are robust with respect to the polynomial degree.
It is also interesting to note that for some of the test cases, the iteration counts \textit{decrease} with increasing $p$;
this may be because the approximation of the Gauss--Lobatto mass matrix by its diagonal improved with increasing $p$ \cite{Kolev2022}.
At lower polynomial degrees, preasymptotic increase in iteration counts is observed with increasing mesh refinement.
With higher polynomial degrees this effect is mitigated.

\begin{table}
   \caption{
      MINRES iteration counts (with relative tolerance of $10^{-12}$) for the Stokes problem.
      The approximate block diagonal preconditioner $\widetilde{\mathcal{B}}$ is used, where the $(1,1)$-block is one of the subspace, fictitious space, and auxiliary space preconditioners.}
   \label{tab:stokes}

   \aboverulesep = 0pt
   \belowrulesep = 0pt

   \resizebox{\linewidth}{!}{%
   \begin{tabular}{|cZ?cc|cc|cc?cc|cc|cc|}
      \toprule
      \tabstrut & \multicolumn{1}{c}{} & \multicolumn{6}{c?}{$p=2\qquad\eta=1$}&\multicolumn{6}{c|}{$p=2\qquad\eta=100$} \\
      \midrule[1pt]
      \multicolumn{2}{|c?}{} & \multicolumn{2}{c|}{$\Bsub A$} & \multicolumn{2}{c}{$\Bfic A$} & \multicolumn{2}{c?}{$\Baux A$} & \multicolumn{2}{c|}{$\Bsub A$} & \multicolumn{2}{c}{$\Bfic A$} & \multicolumn{2}{c|}{$\Baux A$} \\
      $n$ & {\# DOFs} & It. & Time (s) & It. & Time (s) & It. & Time (s) & It. & Time (s) & It. & Time (s) & It. & Time (s) \\
      \midrule[1pt]
      4  &    208 & 114 & 0.001 & 139 & 0.003 & 126 & 0.002 & 278 & 0.003 & 324 & 0.006 & 317 & 0.004 \\
      8  &    800 & 156 & 0.006 & 210 & 0.017 & 187 & 0.010 & 687 & 0.025 & 890 & 0.060 & 839 & 0.040 \\
      16 &   3136 & 196 & 0.028 & 230 & 0.066 & 200 & 0.038 & 961 & 0.141 & 1155 & 0.314 & 1089 & 0.197 \\
      32 &  12416 & 239 & 0.147 & 233 & 0.272 & 202 & 0.153 & 1189 & 0.698 & 1286 & 1.446 & 1178 & 0.837 \\      \midrule[1pt]
      \tabstrut & \multicolumn{1}{c}{} & \multicolumn{6}{c?}{$p=3\qquad\eta=1$}&\multicolumn{6}{c|}{$p=3\qquad\eta=100$} \\
      \midrule[1pt]
      4  &   456 & 124 & 0.004 & 147 & 0.006 & 158 & 0.005 & 421 & 0.012 & 409 & 0.017 & 512 & 0.016 \\
      8  &  1776 & 140 & 0.016 & 161 & 0.029 & 181 & 0.024 & 625 & 0.088 & 616 & 0.107 & 778 & 0.101 \\
      16 &  7008 & 172 & 0.088 & 163 & 0.125 & 189 & 0.104 & 776 & 0.399 & 721 & 0.535 & 920 & 0.480 \\
      32 & 27840 & 208 & 0.438 & 163 & 0.574 & 189 & 0.423 & 956 & 2.012 & 741 & 2.305 & 948 & 2.078 \\
      \midrule[1pt]
      \tabstrut & \multicolumn{1}{c}{} & \multicolumn{6}{c?}{$p=4\qquad\eta=1$}&\multicolumn{6}{c|}{$p=4\qquad\eta=100$} \\
      \midrule[1pt]
      4  &   800 & 117 & 0.009 & 146 & 0.012 & 156 & 0.012 & 434 & 0.033 & 427 & 0.033 & 522 & 0.038 \\
      8  &  3136 & 135 & 0.046 & 153 & 0.050 & 169 & 0.056 & 586 & 0.209 & 585 & 0.192 & 721 & 0.240 \\
      16 & 12416 & 166 & 0.268 & 156 & 0.224 & 174 & 0.236 & 706 & 1.216 & 615 & 0.884 & 785 & 1.097 \\
      32 & 49408 & 201 & 1.353 & 157 & 0.905 & 176 & 1.020 & 864 & 5.814 & 632 & 3.557 & 813 & 4.580 \\
      \midrule[1pt]
      \tabstrut & \multicolumn{1}{c}{} & \multicolumn{6}{c?}{$p=5\qquad\eta=1$}&\multicolumn{6}{c|}{$p=5\qquad\eta=100$} \\
      \midrule[1pt]
      4  &  1240 & 119 & 0.021 & 157 & 0.025 & 157 & 0.022 & 442 & 0.075 & 422 & 0.067 & 544 & 0.076 \\
      8  &  4880 & 133 & 0.126 & 160 & 0.110 & 170 & 0.108 & 553 & 0.538 & 514 & 0.556 & 680 & 0.512 \\
      16 & 19360 & 163 & 0.581 & 161 & 0.509 & 182 & 0.507 & 651 & 2.238 & 545 & 1.563 & 754 & 1.949 \\
      32 & 77120 & 195 & 2.928 & 161 & 1.923 & 187 & 2.026 & 787 & 11.493 & 562 & 6.666 & 788 & 8.294 \\
      \midrule[1pt]
      \tabstrut & \multicolumn{1}{c}{} & \multicolumn{6}{c?}{$p=6\qquad\eta=1$}&\multicolumn{6}{c|}{$p=6\qquad\eta=100$} \\
      \midrule[1pt]
      4  &  1776  & 119 & 0.046 & 167 & 0.046 & 163 & 0.045 & 434 & 0.166 & 445 & 0.125 & 530 & 0.141 \\
      8  &  7008  & 135 & 0.227 & 168 & 0.211 & 172 & 0.204 & 522 & 0.879 & 507 & 0.629 & 624 & 0.744 \\
      16 & 27840  & 165 & 1.153 & 170 & 0.906 & 178 & 0.911 & 621 & 4.483 & 530 & 2.809 & 663 & 3.460 \\
      32 & 110976 & 197 & 5.763 & 171 & 3.738 & 182 & 3.751 & 740 & 21.603 & 541 & 11.856 & 676 & 13.896 \\
      \bottomrule
   \end{tabular}}

\end{table}

\section{Conclusions}
\label{sec:conclusions}

In this paper, we have developed three preconditioners for the interior penalty discontinuous Galerkin discretization in the space $\Hdiv$.
The first preconditioner is a subspace correction preconditioner with vertex patches, where the coarse space is the lowest-order $H^1$-conforming subspace.
The second and third preconditioners are based on auxiliary space techniques, using the standard piecewise polynomial discontinuous Galerkin space as an auxiliary space;
the third preconditioner is distinguished by its use of a block Jacobi-type smoother together with the auxiliary coarse space.
The fictitious space preconditioner is amenable to matrix-free implementation, whereas the other two preconditioners require the solution to small local problems, necessitating the use of assembled matrices and direct solvers for the patch problems.
On Cartesian (or affinely transformed Cartesian) meshes, all three preconditioners exhibit convergence that is independent of discretization parameters $h$, $p$, and $\eta$.
On general meshes, theoretical guarantees of $p$- and $\eta$-independence are not shown, but numerical results indicate that the subspace and auxiliary space preconditioners retain favorable convergence properties also in this case.
The fictitious space preconditioner on general meshes (with skewed elements) is the most sensitive to the penalty parameter.
The application to preconditioning the exactly divergence-free Stokes discretization in $\Hdiv$ is considered.

\section{Acknowledgements}

The author was partially supported by the ORAU Ralph E.\ Powe Junior Enhancement Award and NSF RTG DMS-2136228.

\printbibliography

\end{document}